\documentclass[]{interact}

\usepackage{epstopdf} 
\usepackage{subfigure}
\usepackage{multirow}
\usepackage{float}
\usepackage[numbers,sort&compress]{natbib}
\bibpunct[, ]{[}{]}{,}{n}{,}{,}

\newcommand{\f}{\operatorname}

\newcommand{\R}{\mathbb{R}}
\newcommand{\N}{\mathbb{N}}

\theoremstyle{plain}
\newtheorem{theorem}{Theorem}[section]

\newtheorem{corollary}[theorem]{Corollary}
\newtheorem{proposition}[theorem]{Proposition}

\theoremstyle{definition}
\newtheorem{definition}[theorem]{Definition}

\theoremstyle{remark}

\begin{document}

\articletype{ARTICLE TEMPLATE}

\title{Power laws distributions in objective priors}

\author{ Pedro L. Ramos$^{\rm ab}$$^{\ast}$\thanks{$^\ast$Corresponding author. Email: pedrolramos@usp.br
\vspace{6pt}}, Francisco A. Rodrigues$^{\rm a}$, Eduardo Ramos$^{\rm a}$,  Dipak K. Dey$^{\rm b}$ and Francisco Louzada$^{\rm a}$ \\\vspace{6pt}  $^{a}${Institute of Mathematical Science and Computing, University of  S\~ao \\ Paulo, S\~ao Carlos, Brazil} \\ 
$^{b}${Department of Statistics, University of  Connecticut, Storrs, CT, USA}}

\maketitle

\begin{abstract}
The use of objective prior in Bayesian applications has become a common practice to analyze data without subjective information. Formal rules usually obtain these priors distributions, and the data provide the dominant information in the posterior distribution. However, these priors are typically improper and may lead to improper posterior. Here, we show, for a general family of distributions, that the obtained objective priors for the parameters either follow a power-law distribution or has an asymptotic power-law behavior. As a result, we observed that the exponents of the model are between 0.5 and 1. Understand these behaviors allow us to easily verify if such priors lead to proper or improper posteriors directly from the exponent of the power-law. The general family considered in our study includes essential models such as Exponential, Gamma, Weibull, Nakagami-m, Haf-Normal, Rayleigh, Erlang, and Maxwell Boltzmann distributions, to list a few. In summary, we show that comprehending the mechanisms describing the shapes of the priors provides essential information that can be used in situations where additional complexity is presented.

\end{abstract}

\begin{keywords}
Bayesian inference; objective prior; power-law; statistical method.
\end{keywords}

\section{Introduction}

Bayesian methods have become ubiquitous among statistical procedures and have provided important results in areas from medicine to engineering \cite{lloyd2019improved, wang2019novel}. In the Bayesian approach, the parameters in a statistical model are assumed to be random variables \cite{bernardo2005}, differently from the frequentist approach, that consider these parameters as constant. Moreover, a subjective ingredient can be included in the model, to reproduce the knowledge of a specialist (see O'Hagan et al. \cite{o2006uncertain}). On the other hand, in many situations, we are interested in obtaining a prior distribution, which guarantees that the information provided by the data will not be overshadowed by subjective information. In this case, an objective analysis is recommended by considering non-informative priors that are derived by formal rules \cite{consonni2018prior, kass1996selection}. Although several studies have found weakly informative priors (flat priors) as presumed non-informative priors, Bernardo \cite{bernardo2005} argued that using simple proper priors, supposed to be non-informative, often hides significant unwarranted assumptions, which may easily dominate, or even invalidate the statistical analysis. 

The objective priors are constructed by formal rules \cite{kass1996selection} and are usually improper, i.e., do not correspond to proper probability distribution and could lead to improper posteriors, which is undesirable. According to  Northrop and Attalides \cite{northrop2016}, there are no simple conditions that can be used to prove that improper prior yields a proper posterior for a particular distribution. Therefore a case-by-case investigation is needed to check the propriety of the posterior distribution. The Stacy~\cite{stacy1962} general family of distribution overcomes this problem by proving that if the objective priors follow asymptotically a power-law model with the exponent in some particular regions, then the obtained posteriors are proper or improper.  As a result, one can easily check if the obtained posterior is proper or improper, directly looking at the behavior of the improper prior as a power-law model. 

Understanding the situations when the data follow a power-law distribution can indicate the mechanisms that describe the natural phenomenon in question. Power-law distributions appears in many physical, biological, and man-made phenomena, for instance, they can be used to describe biological network \cite{prvzulj2007biological}, infectious diseases \cite{geilhufe2014power}, the sizes of craters on the moon  \cite{newman2005power}, intensity function in repairable systems \cite{louzada2019repairable} and energy dissipation in cyclones \cite{corral2010scaling} (see also \cite{goldstein2004problems, barrat2008dynamical, newman2018networks}).  The probability density function of a power-law distribution can be represented as
\begin{equation}\label{pld}
\begin{aligned}
\pi(\theta)&=c\,\theta^{-\alpha},
\end{aligned}
\end{equation}
where $c$ is a normalized constant and $\alpha$ the exponent parameter. During the applications of Bayesian methods the normalized constant is usually omitted and the prior can be represented by $\pi(\theta)\propto \theta^{-\alpha}$.

In this paper, we analyze the behavior of different objective priors related to the parameters of many distributions. We show that its asymptotic behavior follows power-law models with exponents between 0.5 and 1. Under these cases, they may lead to proper or improper posterior depending on the exponent values of the priors. Situations, where a power-law distribution is observed with an exponent smaller than one were observed by Goldstein et al. \cite{goldstein2004problems}, Deluca and Corral \cite{deluca2013fitting} and Hanel et al. \cite{hanel2017fitting}. The objective priors are obtained from the Jeffreys' rule \cite{kass1996selection}, Jeffreys' prior \cite{jeffreys1946invariant} and reference priors \cite{bernardo1979a, bernardo2005, berger2015}. Although the posterior distribution may be proper, the posterior moments can be infinite. Therefore, we also provided sufficient conditions to verify if the posterior moments are finite. These results play an important role in which the acknowledgement of the power-law behavior for the prior distribution related to a particular distribution can provide an understanding of the shapes of the prior that can be used in situations where additional complexity (e.g. random censoring, long-term survival, among others) is presented. Priors obtained from formal rules are more difficult or cannot be obtained. 

The remainder of this paper is organized as follows. Section 2 presents the theorems that provide necessary and sufficient conditions for the posterior distributions to be proper depending on the asymptotic behavior of the prior as a power-law model. Additionally, we also discuss sufficient conditions to check if the posterior moments are finite. Sections 3 present study of the behavior of the objective priors. Finally, Section 4 summarizes the study with concluding remarks.

\section{An general model}

The Stacy family of distributions plays an important role in statistics and has proven to be very flexible in practice for modeling data from several areas, such as climatology, meteorology medicine, reliability and image processing data, among others~\cite{stacy1962}. A random variable X follows Stacy's model if its probability density function (PDF) is given by
\begin{equation}\label{denspgg}
f(x|\boldsymbol{\theta})= \alpha \mu^{\alpha\phi}x^{\alpha\phi-1}\exp\left(-(\mu x)^{\alpha}\right)/\Gamma(\phi) , \quad x>0
\end{equation}
where $\Gamma(\phi)=\int_{0}^{\infty}{e^{-x}x^{\phi-1}dx}$ is the gamma function, $\boldsymbol{\theta}=(\phi,\mu,\alpha)$, $\alpha>0$ and $\phi >0$ are the shape parameters and $\mu >0$ is a scale parameter. The Stacy's model unify many important distributions, as shown in Table \ref{t1d}.

\begin{table}[!h]
\caption{Distributions included in the Stacy family of distributions (see equation~\ref{denspgg}).}
\centering 
\begin{center}
  \begin{tabular}{ c | c | c | c}
    \hline
    Distribution  & $\ \ \ \mu \ \ \ $ & $\ \ \ \phi \ \ \ $ & $\ \ \  \alpha \ \ \ $ \\ \hline
        Exponential & $\cdot$  & 1 & 1 \\ 
        Rayleigh & $\cdot$  & 1 & 2 \\
        Haf-Normal & $\cdot$ & 0.5 & 2 \\
        Maxwell Boltzmann & $\cdot$ & $\frac{3}{2}$ & 2 \\
        scaled chi-square & $\cdot$ & 0.5n & 1 \\
        chi-square & 2 & 0.5n & 1 \\
    Weibull     &  $\cdot$  & 1  & $\cdot$ \\ 
        Generalized Haf-Normal  &  $\cdot$  & 2  & $\cdot$ \\ 
        Gamma     &  $\cdot$ & $\cdot$  & $1$ \\ 
        Erlang & $\cdot$ & $n$ & $\cdot$\\
        Nakagami     &  $\cdot$ & $\cdot$  & $2$ \\
        Wilson-Hilferty  &  $\cdot$ & $\cdot$  & $3$ \\ 
        Lognormal & $\cdot$ & $\phi\rightarrow\infty$ & $\cdot$ \\ \hline
  \end{tabular}\label{t1d}
\end{center}
\vspace{-0.55cm}
\begin{flushright}$n\in\N \quad \quad \quad\quad \quad \quad\quad \quad$ \end{flushright}
\end{table}

The inference procedures related to the parameters are conducted using the joint posterior distribution for $\boldsymbol{\theta}$ that is given by the product of the likelihood function and the prior distribution $\pi(\boldsymbol{\theta})$ divided by a normalizing constant $d(\boldsymbol{x})$, resulting in
\begin{equation}\label{posteriord1}
p(\boldsymbol{\theta|x})=\frac{\pi(\boldsymbol{\theta})}{d(\boldsymbol{x})}\frac{\alpha^{n}}{\Gamma(\phi)^n}\left\{\prod_{i=1}^n{x_i^{\alpha\phi-1}}\right\}\mu^{n\alpha\phi}\exp{\left\{-\mu^{\alpha}\sum_{i=1}^n x_i^\alpha\right\}},
\end{equation}
where
\begin{equation}\label{cposteriord1}
d(\boldsymbol{x})=\int\limits_{\mathcal{A}}\pi(\boldsymbol{\theta})\frac{\alpha^{n}}{\Gamma(\phi)^n}\left\{\prod_{i=1}^n{x_i^{\alpha\phi-1}}\right\}\mu^{n\alpha\phi}\exp{\left\{-\mu^{\alpha}\sum_{i=1}^n x_i^\alpha\right\}}d\boldsymbol{\theta}
\end{equation}
and $\mathcal{A}=\{(0,\infty)\times(0,\infty)\times(0,\infty)\}$ is the parameter space of $\boldsymbol{\theta}$. Considering any prior in the form 
$\pi\left(\boldsymbol{\theta}\right)\propto \pi(\mu)\pi(\alpha)\pi(\phi),$
our main aim is to analyze the asymptotic behavior of the priors that leads to power-law distributions allowing to find necessary and sufficient conditions for the posterior to be proper, i.e., $d(\boldsymbol{x})<\infty$.  

In order to study such asymptotic behavior the following definitions and propositions will be useful to prove the results related to the posterior distribution. Let $\overline{\mathbb{R}} = \mathbb{R}\cup \{-\infty, \infty\}$ denote the \textit{extended real number line} with the usual order $(\geq)$, let $\mathbb{R}^+$ denote the positive real numbers and ${\mathbb{R}}_0^+$ denote the positive real numbers including $0$,  and denote $\overline{\mathbb{R}}^+$ and $\overline{\mathbb{R}}_{0}^+ $ analogously. Moreover, if $M\in \mathbb{R}^+$ and $a\in \overline{\mathbb{R}}^+$, we define $M\cdot a$ as the usual product if $a\in\mathbb{R}$, and $M\cdot a=\infty$ if $a=\infty$.

\begin{definition}\label{definition0a}
Let $a\in \overline{\mathbb{R}}_0^+ $ and $b\in \overline{\mathbb{R}}_0^+$. We say that  $a \lesssim b$ if there exist $M\in \mathbb{R}^+$ such that $a\leq M\cdot b$.  If $a \lesssim b$ and $b \lesssim a$ then we say that $a \propto b$. 
\end{definition}

In other words, by the Definition \ref{definition0a} we have that $a \lesssim b$ if either $a<\infty$ or $b=\infty$, and we have that $a\propto b$ if either $a<\infty$ and $b<\infty$, or $a=b=\infty$.

\begin{definition}\label{definition0} Let $\f{g}:\mathcal{U}\to\overline{\mathbb{R}}_0^+$ and $\f{h}:\mathcal{U}\to\overline{\mathbb{R}}_0^+$, where $\mathcal{U}\subset\mathbb{R}$. We say that $\f{g}(x)\lesssim \f{h}(x)$ if there exist $M\in \mathbb{R}^+$ such that $\f{g}(x) \leq M\f{h}(x)$ for every $x\in \mathcal{U}$. If $\f{g}(x) \lesssim \f{h}(x)$ and $\f{h}(x) \lesssim \f{g}(x)$ then we say that $\f{g}(x) \propto \f{h}(x)$.
\end{definition}

\begin{definition}\label{definition1}
Let $\mathcal{U}\subset\mathbb{R}$, $a\in \overline{\mathcal{U}}\cup\{\infty\}$, $\f{g}:\mathcal{U}\to\mathbb{R^+}$ and $\f{h}:\mathcal{U}\to\mathbb{R^+}$. We say that $\f{g}(x)\underset{x\to a}{\lesssim} \f{h}(x)$ if
$\limsup_{x\to a} \dfrac{\f{g}(x)}{\f{h}(x)} < \infty  \,$. If $\f{g}(x)\underset{x\to a}{\lesssim} \f{h}(x)$ and $\f{h}(x)\underset{x\to a}{\lesssim} \f{g}(x)$ then we say that $\f{g}(x)\underset{x\to a}{\propto} \f{h}(x)$.
\end{definition}

The meaning of the relations $\f{g}(x)\underset{x\to a^+}{\lesssim} \f{h}(x)$ and $\f{g}(x)\underset{x\to a^-}{\lesssim} \f{h}(x)$ for $a\in \mathbb{R}$ are defined analogously. Note that, if for some $d\in \mathbb{R}^+$ we have $\lim_{x\to c} \dfrac{\f{g}(x)}{\f{h}(x)} = d$, then it follows directly that $\f{g}(x)\underset{x\to c}{\propto} \f{h}(x)$. The following proposition is a direct consequence of the above definition.

\begin{proposition}\label{properties} Let $a\in \R$, $b\in \overline{\R}$, $c\in[a,b]$, $r\in\R^+$, and let $f_1(x)$, $f_2(x)$, $g_1(x)$ and $g_2(x)$ be continuous functions with domain $(a,b)$ such that $f_1(x)\underset{x\to c}{\lesssim} f_2(x)$ and $g_1(x)\underset{x\to c}{\lesssim} g_2(x)$. Then the following hold
\begin{equation*}
f_1(x)g_1(x)\underset{x\to c}{\lesssim}f_2(x)g_2(x) \quad \mbox{ and } \quad f_1(x)^r\underset{x\to c}{\lesssim}f_2(x)^r.
\end{equation*}

\end{proposition}

 The following proposition relates Definition \ref{definition0} and Definition \ref{definition1}.

\begin{proposition}\label{proportional1}
Let $\f{g}:(a,b)\to\mathbb{R}^+_0$ and $\f{h}:(a,b)\to\mathbb{R^+}$ be continuous functions on $(a,b)\subset\mathbb{R}$, where $a\in\mathbb{R}$ and $b\in\overline{\mathbb{R}}$. Then $\f{g}(x)\lesssim \f{h}(x)$ if and only if $\f{g}(x)\underset{x\to a}{\lesssim} \f{h}(x)$ and $\f{g}(x)\underset{x\to b}{\lesssim} \f{h}(x)$.
\end{proposition}
\begin{proof} See Appendix \ref{propsitionaprof}.\end{proof}

Note that if $\f{g}:(a,b)\to\mathbb{R^+}$ and $\f{h}:(a,b)\to\mathbb{R^+}$ are continuous functions on $(a,b)\subset\mathbb{R}$, then by continuity it follows directly that $\lim_{x\to c} \dfrac{\f{g}(x)}{\f{h}(x)} = \dfrac{\f{g}(c)}{\f{h}(c)} > 0$ and therefore $\f{g}(x)\underset{x\to c}{\propto} \f{h}(x)$ for every $c\in (a,b)$. This fact and the Proposition \ref{proportional1} imply directly the following.

\begin{proposition}\label{proposition1} Let $\f{g}:(a,b)\to\mathbb{R^+}$ and $\f{h}:(a,b)\to\mathbb{R^+}$ be continuous functions in $(a,b)\subset\mathbb{R}$, where $a\in\mathbb{R}$ and $b\in\overline{\mathbb{R}}$, and let $c\in(a,b)$. Then if $\f{g}(x)\underset{x\to a}{\lesssim} \f{h}(x)$ (or $\f{g}(x)\underset{x\to b}{\lesssim} \f{h}(x)$) we have that $\int_a^c g(t)\; dt \lesssim \int_a^c h(t)\; dt$ (respectively $\int_c^b g(t)\; dt \lesssim \int_c^b h(t)\; dt \,$).
\end{proposition}

\subsection{Case when $\alpha$ is known}\label{subsec21}

Let $p(\boldsymbol{\theta|x},\alpha)$ be of the form (\ref{posteriord1}) but considering $\alpha$ fixed and $\boldsymbol{\theta}=(\phi,\mu)$, the normalizing constant is given by
\begin{equation}\label{postthealpha22th}
\begin{aligned}
d(\boldsymbol{x};\alpha)  \propto \int\limits_{\mathcal{A}}\frac{\pi(\boldsymbol{\theta})}{\Gamma(\phi)^n}\left\{\prod_{i=1}^n{x_i^{\alpha\phi-1}}\right\}\mu^{n\alpha\phi}\exp{\left\{-\mu^{\alpha}\sum_{i=1}^n x_i^\alpha\right\}}d\boldsymbol{\theta},
\end{aligned}
\end{equation}
where $\mathcal{A}=\{(0,\infty)\times(0,\infty)\}$ is the parameter space.  Here our purpose reduce to analyze $\pi\left(\boldsymbol{\theta}\right)\propto \pi(\mu)\pi(\phi)$ and find necessary and sufficient conditions for $d(\boldsymbol{x};\alpha)<\infty$.

\begin{theorem}\label{fundteo1alpha}  Suppose that $\pi(\mu,\phi)<\infty$ for all $(\mu,\phi)\in \R_+^2$, that $n\in \N^+$, and suppose that $\pi(\mu,\phi)=\pi(\mu)\pi(\phi)$ and the priors have asymptotic power-law behaviors with
\begin{equation*}
\pi(\mu) \lesssim \mu^{k}, \ \ \  \pi(\phi) \underset{\phi\to 0^+}{\lesssim} \phi^{r_0}\quad\mbox{ and }\quad \pi(\phi) \underset{\phi\to \infty}{\lesssim} \phi^{r_\infty},
\end{equation*}
such that $k = -1$ with $n>-r_0$, or $k >-1$ with $n>-r_0-1$, then $p(\boldsymbol{\theta|x})$ is proper.
\end{theorem}

\begin{proof} See Appendix \ref{ctheoremaalph1}.
\end{proof}

\begin{theorem}\label{fundteo2alpha} Suppose that $\pi(\mu,\phi)>0$ $\forall (\mu,\phi)\in \R_+^2$, $n\in \N^+$, $\pi(\mu,\phi)\gtrsim \pi(\mu)\pi(\phi)$ and the priors have asymptotic power-law behaviors where $\pi(\mu) \gtrsim \mu^{k}$ and one of the following hold:
\begin{itemize}
\item[i)] $k< -1$; or
\item[ii)] $k>-1$ where $\pi(\phi) \underset{\phi\to 0^+}{\gtrsim} \phi^{r_0}$ with $n\leq -r_0-1$; or
\item[iii)] $k=-1$ where $\pi(\phi) \underset{\phi\to 0^+}{\gtrsim} \phi^{r_0}$ with $n\leq-r_0$,
\end{itemize}
then $p(\boldsymbol{\theta|x})$ is improper. 
\end{theorem}

\begin{proof}
See Appendix \ref{ctheoremaalph2}
\end{proof}

\begin{theorem}\label{maintheorem2a}
Let $\pi(\phi,\mu)=\pi(\phi)\pi(\mu)$ and the behavior of $\pi(\mu)$, $\pi(\phi)$ follows asymptotic power-law distributions given by 
\begin{equation*}
\pi(\mu) \propto \mu^{k}, \quad \pi(\phi) \underset{\mu\to 0^+}{\propto} \phi^{r_0}\quad\mbox{ and}\quad \pi(\phi) \underset{\phi\to \infty}{\propto} \phi^{r_\infty},
\end{equation*}
for $k\in \mathbb{R}$, $r_0\in\mathbb{R}$ and $r_\infty\in\mathbb{R}$. The posterior related to $\pi(\phi,\mu)$ is proper if and only if $k = -1$ with $n>-r_0$, or $k >-1$ with $n>-r_0-1$, and in this case the posterior mean of $\phi$ and $\mu$ are finite, as well as all moments.
\end{theorem}

    \begin{proof}
Since the posterior is proper, by Theorem \ref{fundteo1alpha} we have that $k= -1$ with $n>-r_0$ or $k>-1$ with $n>-r_0-1$. 

Let $\pi^*(\phi,\mu)=\phi\pi(\phi,\mu)$. Then $\pi^*(\phi,\mu)=\pi^*(\phi)\pi^*(\mu)$, where $\pi^*(\phi)=\phi\pi(\phi)$ and $\pi^*(\mu)=\pi(\mu)$, and we have
\begin{equation*}
\pi^*(\mu)\propto \mu^{k}, \quad \pi^*(\phi) \underset{\phi\to 0^+}{\propto} \phi^{r_0+1}\quad\mbox{ and}\quad \pi^*(\phi) \underset{\phi\to \infty}{\propto} \phi^{r_\infty+1}.
\end{equation*}

Since $k= -1$ with $n>-r_0>-(r_0+1)$ or $k>-1$ with $n>-(r_0+1)-1$, it follows from Theorem \ref{fundteo1alpha} that the posterior 
\begin{equation*}
\begin{aligned}
\pi^*(\phi,\mu)\frac{\alpha^{n}}{\Gamma(\phi)^n}\left\{\prod_{i=1}^n{x_i^{\alpha\phi-1}}\right\}\mu^{n\alpha\phi}\exp{\left\{-\mu^{\alpha}\sum_{i=1}^n x_i^\alpha\right\}}
\end{aligned}
\end{equation*}
related to the prior $\pi^*(\phi,\mu)$ is proper. Therefore
\begin{equation*}
\begin{aligned}
E[\phi|\boldsymbol{x}]=\int_0^{\infty}\int_0^{\infty}\phi\pi(\phi,\mu)\pi(\boldsymbol{\theta})\frac{\alpha^{n}}{\Gamma(\phi)^n}\left\{\prod_{i=1}^n{x_i^{\alpha\phi-1}}\right\}\mu^{n\alpha\phi}\exp{\left\{-\mu^{\alpha}\sum_{i=1}^n x_i^\alpha\right\}} d\mu d\phi<\infty.
\end{aligned}
\end{equation*}
Analogously one can prove that  
\begin{equation*}
\begin{aligned}
E[\mu|\boldsymbol{x}]=\int_0^{\infty}\int_0^{\infty}\mu\pi(\phi,\mu)\pi(\boldsymbol{\theta})\frac{\alpha^{n}}{\Gamma(\phi)^n}\left\{\prod_{i=1}^n{x_i^{\alpha\phi-1}}\right\}\mu^{n\alpha\phi}\exp{\left\{-\mu^{\alpha}\sum_{i=1}^n x_i^\alpha\right\}} d\mu d\phi<\infty.
\end{aligned}
\end{equation*}

Therefore we have proved that if a prior $\pi(\phi,\mu)$ satisfying the assumptions of the theorem leads to a proper posterior, then the priors $\phi\pi(\phi,\mu)$ and $\mu\pi(\phi,\mu)$ also leads to proper posteriors. It follows by induction that $\phi^r\mu^s\pi(\phi,\mu)$ also leads to proper posteriors for any $r$ and $s \in \N$, which concludes the proof.
\end{proof}

\subsection{Case when $\phi$ is known}\label{sectphiknown}

Let $p(\boldsymbol{\theta|x},\phi)$ be of the form (\ref{posteriord1}) but considering fixed $\phi$ and $\boldsymbol{\theta}=(\mu,\alpha)$, the normalizing constant is given by
\begin{equation}\label{postthephi22th}
\begin{aligned}
d(\boldsymbol{x};\phi)  = \int\limits_{\mathcal{A}}\pi\left(\boldsymbol{\theta}\right)\alpha^{n}\left\{\prod_{i=1}^n{x_i^{\alpha\phi-1}}\right\}\mu^{n\alpha\phi}\exp\left\{-\mu^{\alpha}\sum_{i=1}^n x_i^\alpha\right\}d\boldsymbol{\theta} ,
\end{aligned}
\end{equation}
 where $\mathcal{A}=\{(0,\infty)\times(0,\infty)\}$ is the parameter space.  Let $\pi\left(\boldsymbol{\theta}\right)\propto \pi(\mu)\pi(\alpha)$, our purpose is to find necessary and sufficient conditions where $d(\boldsymbol{x};\phi)<\infty$.  

\begin{theorem}\label{fundteophi1}  Suppose that $\pi(\mu,\alpha)<\infty$ for all $(\mu,\alpha)\in \R_+^2$, that $n\in \N^+$, and suppose that $\pi(\mu,\alpha)=\pi(\alpha)\pi(\mu)$ and the priors have asymptotic power-law behaviors with
\begin{equation*}
\pi(\mu) \lesssim \mu^{k}, \ \ \
\pi(\alpha) \underset{\alpha\to 0^+}{\lesssim} \alpha^{q_0}, \ \ \
\pi(\alpha) \underset{\alpha\to \infty}{\lesssim} \alpha^{q_\infty},
\end{equation*}
such that $k= -1$, $n>-q_0$ and $q_\infty \in \R$. then $p(\boldsymbol{\theta|x})$ is proper.
\end{theorem}

\begin{proof}
 See Appendix \ref{ctheoremphi1}.
\end{proof}

\begin{theorem}\label{fundteophi2} Suppose that $\pi(\alpha,\mu)>0$ $\forall (\alpha,\mu)\in \R_+^2$ and that $n\in \N^+$, and suppose that
 $\pi(\mu,\alpha)\gtrsim \pi(\mu)\pi(\alpha)$ and the priors have asymptotic power-law behaviors where  $\pi(\mu) \gtrsim \mu^{k}$ and one of the following hold
\begin{itemize}
\item[i)] $k< -1$;
\item[ii)] $k > -1$ such that $\pi(\alpha) \underset{\alpha\to 0^+}{\gtrsim} \alpha^{q_0}$ with $q_0\in \R$; or
\item[iii)] $k = -1$ such that $\pi(\alpha) \underset{\alpha\to 0^+}{\gtrsim} \alpha^{q_0}$ with $n\leq -q_0$
\end{itemize}
then $p(\boldsymbol{\theta|x})$ is improper.
\end{theorem}

\begin{proof}
 See Appendix \ref{ctheoremphi2}.
\end{proof}

\begin{theorem}\label{maintheorem2wei}
Let $\pi(\mu,\alpha)=\pi(\mu)\pi(\alpha)$ and the behavior of $\pi(\mu)$, $\pi(\alpha)$  follows asymptotic power-law distributions given by 
\begin{equation*}
\pi(\mu) \propto \mu^{k}, \quad \pi(\alpha) \underset{\mu\to 0^+}{\propto} \alpha^{q_0}\quad\mbox{ and}\quad \pi(\alpha) \underset{\alpha\to \infty}{\propto} \alpha^{q_\infty},
\end{equation*}
for $k\in \mathbb{R}$, $q_0\in\mathbb{R}$ and $q_\infty\in\mathbb{R}$. The posterior related to $\pi(\mu,\alpha)$ is proper if and only if $k = -1$ with $n>-q_0$, and in this case the posterior mean of $\alpha$ is finite for this prior, as well as all moments relative to $\alpha$, and the posterior mean of $\mu$ is not finite.
\end{theorem}

    \begin{proof}
Since the posterior is proper, by Theorem \ref{fundteophi2} we have that $k= -1$ and $n>-q_0$.

Let $\pi^*(\mu,\alpha)=\alpha\pi(\mu,\alpha)$. Then $\pi^*(\mu,\alpha)=\pi^*(\mu)\pi^*(\alpha)$, where $\pi^*(\alpha)=\alpha\pi(\alpha)$ and $\pi^*(\mu)=\pi(\mu)$, and we have
\begin{equation*}
\pi^*(\mu)\propto \mu^{-1}, \quad \pi^*(\alpha) \underset{\mu\to 0^+}{\propto} \alpha^{q_0+1}\quad\mbox{ and}\quad \pi^*(\alpha) \underset{\alpha\to \infty}{\propto} \alpha^{q_\infty+1}.
\end{equation*}

But since $n>-q_0>-(q_0+1)$ it follows from Theorem \ref{fundteophi1} that the posterior 
\begin{equation*}
\begin{aligned}
\pi^*(\mu,\alpha)\frac{\alpha^{n}}{\Gamma(\phi)^n}\left\{\prod_{i=1}^n{x_i^{\alpha\phi-1}}\right\}\mu^{n\alpha\phi}\exp{\left\{-\mu^{\alpha}\sum_{i=1}^n x_i^\alpha\right\}}
\end{aligned}
\end{equation*}
relative to the prior $\pi^*(\mu,\alpha)$ is proper. Therefore
\begin{equation*}
\begin{aligned}
E[\alpha|\boldsymbol{x}]=\int_0^{\infty}\int_0^{\infty}\alpha\pi(\mu,\alpha)\pi(\boldsymbol{\theta})\frac{\alpha^{n}}{\Gamma(\phi)^n}\left\{\prod_{i=1}^n{x_i^{\alpha\phi-1}}\right\}\mu^{n\alpha\phi}\exp{\left\{-\mu^{\alpha}\sum_{i=1}^n x_i^\alpha\right\}} d\mu d\alpha<\infty.
\end{aligned}
\end{equation*}
Analogously one can prove using the item ii) of the Theorem \ref{fundteophi2}  that  
\begin{equation*}
\begin{aligned}
E[\mu|\boldsymbol{x}]=\int_0^{\infty}\int_0^{\infty}\mu\pi(\mu,\alpha)\pi(\boldsymbol{\theta})\frac{\alpha^{n}}{\Gamma(\phi)^n}\left\{\prod_{i=1}^n{x_i^{\alpha\phi-1}}\right\}\mu^{n\alpha\phi}\exp{\left\{-\mu^{\alpha}\sum_{i=1}^n x_i^\alpha\right\}} d\mu d\alpha=\infty
\end{aligned}
\end{equation*}
since in this case $\mu\pi(\mu)\propto \mu^{0}$.

Therefore we have proved that if a prior $\pi(\mu,\alpha)$ satisfying the assumptions of the theorem leads to a proper posterior, then the prior $\alpha\pi(\mu,\alpha)$ also leads to proper posteriors. It follows by induction that $\alpha^r\pi(\mu,\alpha)$ also leads to proper posteriors for any $r$ in $\N$, which concludes the proof.
\end{proof}

\subsection{General case when $\phi$, $\alpha$ and $\mu$ are unknown}

\begin{theorem}\label{fundteo1}  Suppose that $\pi(\alpha,\beta,\mu)<\infty$ for all $(\alpha,\beta,\mu)\in \R_+^3$, that $n\in \N^+$, and suppose that $\pi(\mu,\alpha,\phi)=\pi(\mu)\pi(\alpha)\pi(\mu)$ and the priors have asymptotic power-law behaviors with
\begin{equation*}
\pi(\mu) \lesssim \mu^{k},   \ \ \ 
\pi(\alpha) \underset{\alpha\to 0^+}{\lesssim} \alpha^{q_0}, \ \ \ \pi(\alpha) \underset{\alpha\to \infty}{\lesssim} \alpha^{q_\infty}, 
\end{equation*}
\vspace{-0.2cm}
\begin{equation*}
\pi(\phi) \underset{\phi\to 0^+}{\lesssim} \phi^{r_0}\quad\mbox{ and }\quad \pi(\phi) \underset{\phi\to \infty}{\lesssim} \phi^{r_\infty},
\end{equation*}
such that $k= -1$, $q_\infty < r_0 $, $ 2r_\infty+1 < q_0$, $n>-q_0$ and $n>-r_0$, then $p(\boldsymbol{\theta|x})$ is proper.
\end{theorem}
\begin{proof}
See Appendix \ref{ctheoremab2} 
\end{proof}

\begin{theorem}\label{fundteo2} Suppose that $\pi(\alpha,\phi,\mu)>0$ $\forall (\alpha,\phi,\mu)\in \R_+^3$ and that $n\in \N^+$, then the following items are valid
\begin{itemize}
\item[i)] $\pi(\mu,\alpha,\beta)\gtrsim \pi(\mu)\pi(\alpha)\pi(\phi)$ for all $\phi\in [b_0,b_1]$ where $0\leq b_0< b_1$, such that $\pi(\mu)  \gtrsim \mu^{k}$ and one of the following hold
\begin{itemize}
\item[-]  $k< -1$;
\item[-]  $k>-1$; where $\pi(\alpha) \underset{\alpha\to 0^+}{\gtrsim} \alpha^{q_0}$ with $q_0\in \R$; or
\item[-] $k > -1$; where $\pi(\phi) \underset{\phi\to 0^+}{\gtrsim} \phi^{r_0}$ with $n<-r_0-1$ and $b_0=0$.
\end{itemize}
then $p(\boldsymbol{\theta|x})$ is improper.

\item[ii)] $\pi(\mu,\alpha,\beta)\gtrsim \pi(\mu)\pi(\alpha)\pi(\beta)$ such that $\pi(\mu)  \gtrsim \mu^{-1}$ and one of the following occur
\begin{itemize}
\item[-] $\pi(\phi) \underset{\phi\to 0^+}{\gtrsim} \phi^{r_0}$ and $\pi(\alpha)
\underset{\alpha\to \infty}{\gtrsim} \alpha^{q_\infty}$ where either $q_\infty \geq r_0$ or $n\leq -r_0$;
\item[-] $\pi(\alpha) \underset{\alpha\to 0^+}{\gtrsim} \alpha^{q_0}$ and $ \pi(\phi) \underset{\phi\to \infty}{\gtrsim} \phi^{r_\infty}$ where either $ 2r_\infty+1 \geq q_0$ or $n\leq -q_0$;
\end{itemize}
then $p(\boldsymbol{\theta|x})$ is improper.
\end{itemize}
\end{theorem}
\begin{proof}
See Appendix \ref{ctheoremab3} 
\end{proof}

\begin{theorem}\label{fundteo3} Suppose that $0 < \pi(\alpha,\beta,\mu)<\infty$ for all $(\alpha,\beta,\mu)\in \R_+^3$, and suppose that $\pi(\mu,\alpha,\phi)=\pi(\mu)\pi(\alpha)\pi(\phi)$ where the priors have asymptotic power-law behaviors with
\begin{equation*}
\pi(\mu) \propto \mu^{k}, \ \ \ \pi(\alpha) \underset{\alpha\to 0^+}{\propto} \alpha^{q_0},  \ \ \  \pi(\alpha) \underset{\alpha\to \infty}{\propto} \alpha^{q_\infty},
\end{equation*}
\begin{equation*}
\pi(\phi) \underset{\phi\to 0^+}{\propto} \phi^{r_0}\quad\mbox{ and }\quad \pi(\phi) \underset{\phi\to \infty}{\propto} \phi^{r_\infty},
\end{equation*}
then the posterior is proper if and only if $k= -1$, $q_\infty < r_0 $, $ 2r_\infty+1 < q_0$, $n>-q_0$ and $n>-r_0$. Moreover, if the posterior is proper then $\alpha^q\phi^r\mu^j\pi(\alpha,\phi,\mu)$ leads to a proper posterior if and only if $j=0$, and $2(r+r_\infty)+1-q_0<q<r+r_0-q_\infty$.
\end{theorem}

\begin{proof} Notice that under our hypothesis, Theorems \ref{fundteo2} and \ref{fundteo3} are complementary, and thus the first part of the theorem is proved. Analogously, by the Theorems \ref{fundteo2} and \ref{fundteo3} the prior $\alpha^q\beta^r\mu^l\pi(\alpha,\beta,\mu)$ leads to a proper posterior if and only if $j= 0$, $q+q_\infty<r+r_0$, $2(r+r_\infty)+1 < q+q_0$, $n> -q_0-q$ and $n> -r_0-r$. The last two proportionalities are already satisfied since $n>-q_0$ and $n>-r_0$. Combining the other inequalities the proof is completed.
\end{proof}

\section{Some common objective priors with power-law asymptotic behavior}

A common approach was suggested by Jeffreys' that considered different procedures for constructing objective priors. For $\theta\in(0,\infty)$ (see, \cite{kass1996selection}), Jeffreys suggested to use the prior $\pi(\theta)=\theta^{-1}$, i.e., a power-law distribution with exponent 1. The main justification for this choice is its invariance under power transformations of the parameters. As the parameters of the Stacy family of distributions are contained in the interval $(0,\infty)$, the prior using Jeffreys' first rule is $\pi_1\left(\phi,\mu,\alpha\right)\propto (\phi\mu\alpha)^{-1}$.

Let us consider the case when $\alpha$ is known. Hence, the results is valid for the Gamma, Nakagami, Wilson-Hilferty distributions, among others. The  Jeffreys' first rule when $\alpha$ is known follows power-law distributions with $\pi(\phi)\propto\phi^{-1}$ and $\pi(\mu)\propto\mu^{-1}$. Hence the posterior distribution obtained is proper for all $n>1$ as well as its higher moments. This can be easily proved by noticing that as $\pi_1(\phi,\mu)\propto\phi^{-1}\mu^{-1}$ we can apply Theorem \ref{maintheorem2wei} with $k=r_0=r_\infty=-1$ and it follows that the posterior is proper for $n>-r_0=1$ as well as its moments.

On the other hand, under the general model where all the parameters are unknown, we have the  posterior distribution (\ref{posteriord1}) obtained using Jeffreys' first rule is improper for all $n\in\N^+$.  Since $\pi(\phi)\propto\phi^{-1}$, $\pi(\alpha)\propto\alpha^{-1}$ and $\pi(\mu)\propto\mu^{-1}$, i.e., power-laws with exponent $1$, we can apply Theorem \ref{fundteo2} ii) with $k=q_\infty=r_0=-1$, where $q_\infty\geq r_0$, and therefore we have that $\pi_2(\alpha,\beta,\mu)\propto\phi^{-1}\alpha^{-1}\mu^{-1}$ leads to an improper posterior for all $n\in\N^+$. 

Let us consider the cases  where  $\pi(\mu)\propto\mu^{-1}$ and the $\pi(\phi)$ has different forms which can be written as
\begin{equation}\label{priorjgenje}
\pi_j\left(\boldsymbol{\theta}\right)\propto \frac{\pi_j(\phi)}{\mu},
\end{equation}
where $j$ is the index related to a particular prior. Therefore, our main focus will be to study the behavior of the priors $\pi_j(\phi)$.

One important objective prior is based on Jeffreys' general rule \cite{jeffreys1946invariant} and known as Jeffreys' prior. This prior is obtained through the square root of the determinant of the Fisher information matrix and has been widely used due to its invariance property under one-to-one transformations. The Fisher information matrix for the Stacy family of distributions was derived by \cite{hager1970inferential} and its elements are given by
\begin{equation*}
I_{\alpha,\alpha}(\boldsymbol{\theta})= \dfrac{1+2\psi(\phi)+\phi\psi ' (\phi)+\phi\psi(\phi)^2}{\alpha^2}, \ I_{\alpha,\mu}(\boldsymbol{\theta})= -\dfrac{\psi(\phi)}{\alpha}, \ I_{\mu,\phi}(\boldsymbol{\theta})=\dfrac{\alpha}{\mu}, 
\end{equation*}
\begin{equation*}
 I_{\alpha,\phi}(\boldsymbol{\theta})=  -\dfrac{1+\phi\psi(\phi)}{\mu}, \ \ I_{\mu,\mu}(\boldsymbol{\theta})=\dfrac{\phi\alpha^2}{\mu^2} \ \ \mbox{ and } \ \ I_{\phi,\phi}(\boldsymbol{\theta})= \psi ' (\phi) ,
\end{equation*}
where  $\psi'(k)=\frac{\partial}{\partial k}\psi(k)$ is the trigamma function.

Van Noortwijk \cite{van2001bayes} provided the Jeffreys' prior for the general model, which can be expressed by (\ref{priorjgenje}) with
\begin{equation}\label{priorjgg}
\pi_3\left(\phi\right)\propto \sqrt{\phi^2\psi^{'}(\phi)^2-\psi^{'}(\phi)-1}.
\end{equation}

\begin{corollary}\label{corrolpriorje} The prior $\pi_3\left(\phi\right)$ has the asymptotic behavior given by 
\begin{equation*}
\pi_3\left(\phi\right)\underset{\phi\to 0^{+}}{\propto} \phi^{0} \quad \mbox{ and } \quad \pi_3\left(\phi\right)\underset{\phi\to \infty}{\propto} \phi^{-1}, 
\end{equation*}
then the obtained posterior distribution  is improper for all $n\in\N^+$.
\end{corollary} 
\begin{proof} Ramos et al. \cite{ramos2017bayesian} proved that
\begin{equation}\label{equatinpropi}
\sqrt{\phi^2\psi^{'}(\phi)^2-\psi^{'}(\phi)-1}\underset{\phi\to 0^{+}}{\propto}  1\ \mbox{ and }\ \sqrt{\phi^2\psi^{'}(\phi)^2-\psi^{'}(\phi)-1}\underset{\phi\to \infty}{\propto} \frac{1}{\phi}.
\end{equation}
Since $\pi_3\left(\phi\right)\underset{\phi\to 0^{+}}{\propto} 1$, the hypotheses of Theorem \ref{fundteo2}, ii) hold with $k=-1$ and $r_0=q_\infty=0$, where $q_\infty \geq r_0$, and therefore $\pi_3(\boldsymbol{\theta})$ leads to an improper posterior for all $n\in \N^+$.
\end{proof}

 Let $\alpha$ be known, then the Jeffreys' prior has the form  (\ref{priorjgenje}) where $\pi(\phi)$ is given by
\begin{equation}\label{eqprio1phip}
\pi_4(\phi)\propto\sqrt{\phi\psi'(\phi)-1}.
\end{equation}

\begin{corollary} 
 The prior $\pi_4\left(\phi\right)$ has the asymptotic power-law behavior given by 
\begin{equation*}
\pi_4\left(\phi\right)\underset{\phi\to 0^{+}}{\propto} \phi^{-\frac{1}{2}} \quad \mbox{ and } \quad \pi_4\left(\phi\right)\underset{\phi\to \infty}{\propto} \phi^{-\frac{1}{2}}, 
\end{equation*} then the obtained posterior is proper for $n\geq 1$ as well as its higher moments.
\end{corollary} 
\begin{proof} Here, we have $\pi(\beta)=\beta^{-1}$, i.e, power-law distribution.
Following \cite{abramowitz} we have that $\lim_{z\to 0^+} \dfrac{\psi'(z)}{z^{-2}}=1$, then $
\lim_{\phi\to 0^+}  \dfrac{\phi\psi'(\phi) - 1}{\phi^{-1}} = \lim_{\phi\to 0^+} \dfrac{\psi'(\phi)}{\phi^{-2}} - \phi = 1$, and thus
\begin{equation}\label{neweqdm134}
\phi\psi'(\phi) - 1 \underset{\phi\to 0^+}{\propto} \phi^{-1},
\end{equation}
which implies $\sqrt{\phi\psi'(\phi) - 1} \underset{\phi\to 0^+}{\propto} \phi^{-\frac{1}{2}}$.
 Moreover, from \cite{abramowitz},  we have that $\psi'(z) = \dfrac{1}{z} + \dfrac{1}{2z^2} + o\left(\dfrac{1}{z^3}\right)$ and thus
\begin{align*}
\frac{\phi\psi'(\phi) - 1}{\phi^{-1}}  = \frac{1}{2} + o\left(\frac{1}{\phi}\right)
\Rightarrow \lim_{\phi\to\infty} \frac{\sqrt{\phi\psi'(\phi) - 1}}{\phi^{-\frac{1}{2}}} = \frac{1}{\sqrt{2}},
\end{align*}
which implies $\sqrt{\phi\psi'(\phi) - 1} \underset{\phi\to \infty}{\propto} \phi^{-\frac{1}{2}}$.

Therefore we can apply Theorem \ref{maintheorem2a} with $k=-1$ and $r_0=r_\infty=-\frac{1}{2}$ and therefore the posterior is proper and the posterior moments are finite for all $n>-r_0=\frac{1}{2}$.
\end{proof}

Fonseca et al. \cite{fonseca2008objective} considered the scenario where the Jeffreys' prior has an independent structure, i.e., the prior has the form $\pi_{J2}\left(\boldsymbol{\theta}\right)\propto\sqrt{|\f{diag}I(\boldsymbol{\theta})|}$, where diag$\,I(\cdot)$ is the diagonal matrix of $I(\cdot)$. For the general distribution the prior is given by (\ref{priorjgenje}) with 
\begin{equation}\label{priorjgg2a}
\pi_4\left(\phi\right)\propto \sqrt{\phi\psi^{'}(\phi)\left(1+2\psi(\phi)+\phi\psi^{'}(\phi)+\phi\psi(\phi)^2\right)}.
\end{equation}

Notice that for (\ref{priorjgg2a}) is only necessary to know the behavior $\pi_4\left(\phi\right)$ when $\phi\to 0^+$ that provided enough information to very that the posterior is improper.

\begin{corollary} The prior (\ref{priorjgg2a}) has the asymptotic power-law behavior given by $\pi_4\left(\phi\right)\underset{\phi\to 0^+}{\propto} \phi^{-\frac{1}{2}}$ and the obtained posterior is improper for all $n\in\N^+$.
\end{corollary} 
\begin{proof} By Abramowitz  and Stegun\cite{abramowitz}, we have the recurrence relations
\begin{equation}\label{digammato0}
\psi(\phi)=-\frac{1}{\phi} +\psi(\phi+1) \ \ \mbox{ and } \ \ \psi'(\phi)=\frac{1}{\phi^2} +\psi'(\phi+1).
\end{equation}

It follows that 
\begin{equation*}
\begin{aligned}
& 2\psi(\phi)+\phi\psi'(\phi)+\phi\psi(\phi)^2+1 = \\
& 2\left(-\frac{1}{\phi}+\psi(\phi+1) \right)+\phi\left(\frac{1}{\phi^2}+\psi'(\phi+1) \right)+\phi\left(\frac{1}{\phi^2}-\frac{2}{\phi}\psi(\phi+1)+\psi(\phi+1)^2 \right)+1 = \\
& 1+\phi\left(\psi(\phi+1)^2+\psi'(\phi+1) \right).
\end{aligned}
\end{equation*}
Hence, $ 2\psi(\phi)+\phi\psi'(\phi)+\phi\psi(\phi)^2+1 \underset{\phi\to 0^+}{\propto} 1$, which implies that
\begin{equation}\label{complicatedto0}
\begin{aligned}
\pi_4\left(\phi\right)\propto\sqrt{\phi\psi^{'}(\phi)\left(1+2\psi(\phi)+\phi\psi^{'}(\phi)+\phi\psi(\phi)^2\right)} \underset{\phi\to 0^+}{\propto} \phi^{-\frac{1}{2}},
\end{aligned}
\end{equation}
i.e., power-law distribution with exponent $\frac{1}{2}$, then, Theorem \ref{fundteo2} ii) can be applied with $k=-1$, $r_0 = -\frac{1}{2}$ and $q_\infty = 0$ where $q_\infty \geq r_0$ and therefore $\pi_4(\boldsymbol{\theta})$ leads to an improper posterior.
\end{proof}



This approach can be further extended considering that only one parameter is independent. For instance, let $(\theta_1,\theta_2)$ be dependent parameters and $\theta_3$ be independent then under the partition the $((\theta_1,\theta_2),\theta_3)$-Jeffreys' prior is given by
\begin{equation}\label{priorjgg2}
\pi\left(\boldsymbol{\theta}\right)\propto \sqrt{\left(I_{11}(\boldsymbol{\theta})I_{22}(\boldsymbol{\theta})-I_{12}^2(\boldsymbol{\theta})\right)I_{33}(\boldsymbol{\theta})}.
\end{equation}
For the general model the partition $((\phi,\mu),\alpha)$-Jeffreys' prior is of the form (\ref{priorjgenje}) with 
\begin{equation}\label{priorjgg3}
\pi_5\left(\phi\right)\propto \sqrt{(\phi\psi^{'}(\phi)-1)\left(1+2\psi(\phi)+\phi\psi^{'}(\phi)+\phi\psi(\phi)^2\right)}.
\end{equation}

\begin{corollary}  The prior (\ref{priorjgg3}) has the asymptotic power-law behavior given by $\pi_5\left(\phi\right)\underset{\phi\to 0^+}{\propto} \phi^{-\frac{1}{2}}$ and the obtained posterior is improper for all $n\in\N^+$.
\end{corollary} 
\begin{proof} From equation (\ref{neweqdm134}) we have that $\phi\psi^{'}(\phi) -1 \underset{\phi\to 0^+}{\propto} \frac{1}{\phi}$ which combined with the relation (\ref{complicatedto0}) implies that
\begin{equation}
\pi_5\left(\phi\right)\propto \sqrt{(\phi\psi^{'}(\phi)-1)\left(1+2\psi(\phi)+\phi\psi^{'}(\phi)+\phi\psi(\phi)^2\right)} \underset{\phi\to 0^+}{\propto} \phi^{-\frac{1}{2}}.
\end{equation}
i.e., power-law distribution with exponent $\frac{1}{2}$,  then Theorem \ref{fundteo2}, ii) can be applied with $k=-1$, $r_0 = -\frac{1}{2}$ and $q_\infty = 0$ where $q_\infty \geq r_0$ and therefore $\pi_5(\boldsymbol{\theta})$ leads to an improper posterior.
\end{proof}

Considering the partition $((\alpha,\mu),\phi)$-Jeffreys' prior is given by (\ref{priorjgenje}) where 
\begin{equation}\label{priorjgg5}
\pi_6\left(\phi\right)\propto \sqrt{\psi^{'}(\phi)(\phi^2\psi^{'}(\phi)+\phi-1)}\,.
\end{equation}

Similar to the two cases above. From the recurrence relations (\ref{digammato0}), we have that 
\begin{equation}\label{eqaux2}
\phi^2\psi^{'}(\phi)+\phi-1 = \phi\left(1 + \phi \psi^{'}(\phi+1)\right) \Rightarrow \phi^2\psi^{'}(\phi)+\phi-1 \underset{\phi\to 0^+}{\propto} \phi
\end{equation}
as $\psi'(\phi)\propto \frac{1}{\phi^2}$ it follows that
\begin{equation*}
\pi_6\left(\phi\right)\propto \sqrt{\psi^{'}(\phi)(\phi^2\psi^{'}(\phi)+\phi-1)} \underset{\phi\to 0^+}{\propto} \phi^{-\frac{1}{2}},
\end{equation*}
with the same values $k=-1$, $r_0 = -\frac{1}{2}$ and $q_\infty = 0$ where $q_\infty \geq r_0$, the prior $\pi_6(\boldsymbol{\theta})$ leads to an improper posterior.

Another important class of objective priors was introduced by Bernardo \cite{bernardo1979a} with further developments \cite{berger1989estimating, berger1992ordered,berger1992development} reference priors play an important role in objective Bayesian analysis. The reference priors have desirable properties, such as invariance, consistent marginalization, and consistent sampling properties. \cite{bernardo2005} reviewed different procedures to derive reference priors considering ordered parameters of interest. The following proposition will be applied to obtain reference priors for the Generalized Gamma distribution.

\begin{proposition}\label{propositionop} [ Bernardo \cite{bernardo1979a}, pg 40, Theorem 14] Let $\boldsymbol\theta=(\theta_1,\ldots,\theta_m)$ be a vector with the ordered parameters of interest and $p(\boldsymbol\theta|\boldsymbol{x})$ be the posterior distribution that has an asymptotically normal distribution with dispersion matrix $V(\hat{\boldsymbol\theta}_n)/n$, where $\hat{\boldsymbol\theta}_n$ is a consistent estimator of $\boldsymbol\theta$ and $H(\boldsymbol\theta)=V^{-1}(\boldsymbol\theta)$. In addition, $V_j$ is the upper $j\times j$ submatrix of $V$, $H_j=V_j$ and $h_{j,j}(\boldsymbol\theta)$ is the lower right element of $H_j$. If the parameter space of $\theta_{j}$ is independent of $\boldsymbol\theta_{-j}=(\theta_1,\ldots,\theta_{j-1},\theta_{j+1},\ldots,\theta_{m}),$ for $j=1,\ldots,m,$  and $h_{j,j}(\boldsymbol\theta)$ are factorized in the form 
$h_{j,j}^{\frac{1}{2}}(\boldsymbol\theta)=f_j(\theta_j)g_j(\boldsymbol\theta_{-j}), \quad j=1,\ldots,m$, then the reference prior for the ordered parameters $\boldsymbol\theta$ is given by 
\begin{equation*}
\pi(\boldsymbol\theta)=\pi(\theta_j|\theta_1,\ldots,\theta_{j-1})\times\cdots\times\pi(\theta_2|\theta_1)\pi(\theta_1),
\end{equation*}
where $\pi(\theta_j|\theta_1,\ldots,\theta_{j-1})=f_j(\theta_j),$ for $j=1,\ldots,m$,  and there is no need for compact approximations, even if the conditional priors are not proper. 
\end{proposition} 

The reference priors obtained from Proposition \ref{propositionop} belong to the class of improper priors given by
\begin{equation}\label{priorjgenref}
\pi\left(\boldsymbol{\theta}\right)\propto \pi(\phi)\alpha^{-1}\mu^{-1},
\end{equation}
therefore, both $\pi(\mu)\propto\mu^{-1}$, $\pi(\alpha)\propto\alpha^{-1}$ follows power-law distributions with exponent $1$. Our focus will be study the asymptotic power-law behavior of $\pi(\phi)$. Let $(\alpha,\phi,\mu)$ be the ordered parameters of interest, then conditional priors of the $(\alpha,\phi,\mu)$-reference prior are given by
\begin{equation*}
\pi(\alpha)\propto \alpha^{-1}, \ \ \ \  \pi(\phi|\alpha)\propto\sqrt{\frac{\phi\psi'(\phi)-1}{\phi}}, \ \ \ \ \pi(\mu|\alpha,\phi)\propto\mu^{-1} .
\end{equation*}

Therefore, $(\alpha,\phi,\mu)$-reference prior is of the form (\ref{priorjgenref}) with
\begin{equation*}
\pi_7(\phi)\propto \sqrt{\frac{\phi\psi'(\phi)-1}{\phi}} \underset{\phi \to 0^+}{\propto} \phi^{-1} \cdot
\end{equation*}
which is also a power-law distribution with exponent $-1$. Therefore, item ii) of Theorem \ref{fundteo2} can be applied with $k= r_0 = q_\infty = 1$ where $q_\infty \geq r_0$ which implies that $\pi_7(\alpha,\phi,\mu)$ leads to an improper posterior for all $n\in \N^+$.

Assuming that $(\alpha,\mu,\phi)$ are the ordered parameters, then the conditional reference priors are
\begin{equation*}
\pi(\alpha)\propto \alpha^{-1}, \ \ \ \  \pi(\mu|\alpha)\propto \mu^{-1}, \ \ \ \ \pi(\phi|\alpha,\mu)\propto\sqrt{\psi'(\phi)},
\end{equation*}
and the $(\alpha,\mu,\phi)$-reference prior is of the form (\ref{priorjgenref}) with
\begin{equation*}
\pi_8(\phi)\propto \sqrt{\psi'(\phi)}.
\end{equation*}

From $\psi'(\phi)\underset{\phi\to 0^+}{\propto} \phi^{-2}$ we have that $\sqrt{\psi'(\phi)} \underset{\phi\to 0^+}{\propto} \phi^{-1}$, i.e., a PL distribution with exponent $-1$. Similar to the case of $\pi_7(\alpha,\phi,\mu)$ we have that $\pi_8(\alpha,\mu,\phi)$ leads to an improper posterior for all $n\in\N^+$.

Consider the case where $\alpha$ is known with $\alpha=1$ reducing to the Gamma distribution. Then $\pi(\phi,\mu)\propto\mu^{-1}\sqrt{\psi'(\phi)}$ is the $(\mu,\phi)$-reference prior and the joint posterior densities when $\alpha=1$  using the $(\mu,\phi)$-reference is proper for $n\geq 2$ as well as its higher moments.

The results above follows from the fact that 
$\psi'(\phi) \underset{\phi\to 0^+}{\propto} \phi^{-2}$ and $\psi'(\phi) \underset{\phi\to \infty^+}{\propto} ^{-1}$ and thus $\pi_8(\phi)$ has asymptotic power-law behavior given by
\begin{equation*}
\pi_8(\phi) \underset{\phi\to 0^+}{\propto} \phi^{-1} \quad \mbox{and} \quad \pi_8(\phi) \underset{\phi\to \infty^+}{\propto} \phi^{-\frac{1}{2}},
\end{equation*}
therefore, from the power-law distributions above as well as the distribution $\pi(\mu)$ that has a PL with exponent 1, we can apply Theorem \ref{maintheorem2a} with $k=-1$, $r_0=-1$ and $r_\infty=-0.5$ and it follows that the posterior as well as all its moments are proper for all $n>-r_0=1$.

Assuming now that $\phi$ is known with $\phi=1$, then the distribution reduces to the Weibull distribution. In this case, $\pi(\mu,\alpha)\propto\alpha^{-1}\mu^{-1}$ is the $(\alpha,\mu)$-reference prior, note that each prior follows a power-law distribution. The joint posterior densities using the $(\alpha,\mu)$-reference is proper for $n\geq 2$ although its higher moments relative to $\mu$ are improper. This result is a direct consequence from Theorem \ref{maintheorem2wei} considering that $k=-1$ and $q_0=q_\infty=-1$ that leads to a proper posterior.

Returning to general model, if $(\mu,\phi,\alpha)$ is the vector of ordered parameters, we have that the conditional priors are
\begin{equation*}
\pi(\mu)\propto \mu^{-1}, \ \  \pi(\phi|\mu)\propto \sqrt{\psi'(\phi)-\frac{\psi(\phi)^2}{2\psi(\phi)+\phi\psi'(\phi)+\phi\psi(\phi^2)+1}}, \  \ \pi(\alpha|\phi,\mu)\propto\alpha^{-1}
\end{equation*}
and the $(\mu,\phi,\alpha)$-reference prior is of the form (\ref{priorjgenref}) with
\begin{equation*}
\pi_9(\phi)\propto \sqrt{\psi'(\phi)-\frac{\psi(\phi)^2}{2\psi(\phi)+\phi\psi'(\phi)+\phi\psi(\phi^2)+1}} \cdot
\end{equation*}

\begin{corollary}\label{postimreference2}  The prior $\pi_9(\phi)$ has the asymptotic power-law behavior given by $\pi_9\left(\phi\right)\underset{\phi\to 0^+}{\propto} \phi^{-1}$ and the obtained posterior is improper for all $n\in\N^+$.
\end{corollary}
\begin{proof}
From \cite{abramowitz}, we have
\begin{equation}\label{digammatoinfty}
\psi(\phi)=\log(\phi)-\frac{1}{2\phi}-\frac{1}{12\phi^2}+o\left(\frac{1}{\phi^2}\right) \ \ \mbox{ and } \ \
\psi'(\phi)=\frac{1}{\phi}+\frac{1}{2\phi^2}+o\left(\frac{1}{\phi^2}\right),
\end{equation}
where it follows directly that
\begin{equation*}
\psi(\phi)^2=\log(\phi)^2-\frac{\log(\phi)}{\phi} + o\left(\frac{1}{\phi} \right).
\end{equation*}

Therefore $2\psi(\phi)+\phi\psi'(\phi)+\phi\psi(\phi)^2+1=\phi\log(\phi)^2+\log(\phi)+2+o(1)$ and
\begin{equation*}
\begin{aligned}
\pi_9(\phi)&\propto \sqrt{\psi'(\phi)- \frac{\psi(\phi)^2}{2\psi(\phi)+\phi\psi'(\phi)+\phi\psi(\phi)^2+1}} \\
&=\sqrt{\frac{\left(\frac{1}{\phi}+\frac{1}{2\phi^2}+o\left(\frac{1}{\phi^2}\right)\right)\left(\phi\log(\phi)^2+\log(\phi)+2+o(1)\right)-\log(\phi)^2+\frac{\log(\phi)}{\phi}+o\left(\frac{1}{\phi} \right)}{\phi\log(\phi)^2+\log(\phi)+2+o(1)}}\\
& =\sqrt{\frac{\frac{1}{\phi}\left(\log(\phi)^2+o(\log(\phi)^2) \right)}{\phi\left(\log(\phi)^2+o(\log(\phi)^2) \right)}} = \frac{1}{\phi}\sqrt{\frac{1+o(1)}{1+o(1)}}.
\end{aligned}
\end{equation*}

Thus
\begin{equation*}
\begin{aligned}
\pi_9(\phi)\propto    \sqrt{\psi'(\phi)- \frac{\psi(\phi)^2}{2\psi(\phi)+\phi\psi'(\phi)+\phi\psi(\phi)^2+1}} \underset{\phi\to 0^+}{\propto} \phi^{-1},
\end{aligned}
\end{equation*}
and therefore Theorem \ref{fundteo2} ii) can be applied with $k= q_0 = r_\infty = -1$ where $2r_\infty+1 \geq q_0$. Thus, $\pi_9(\boldsymbol{\theta})$ leads to an improper posterior.

\end{proof}

Finally, let $(\phi,\alpha,\mu)$ be the ordered parameters, then the conditional priors are
\begin{equation*}
\pi(\phi)\propto\sqrt{\frac{\phi^2\psi'(\phi)^2-\psi'(\phi)-1}{\phi^2\psi'(\phi)+\phi-1}},  \ \ \ \ \pi(\alpha|\phi)\propto\alpha^{-1}, \ \ \ \  \pi(\mu|\alpha,\phi)\propto\mu^{-1} 
\end{equation*}
and the $(\phi,\alpha,\mu)$-reference prior is of the form (\ref{priorjgenref}) with
\begin{equation}\label{priorphip}
\pi_{10}(\phi)\propto \sqrt{\frac{\phi^2\psi'(\phi)^2-\psi'(\phi)-1}{\phi^2\psi'(\phi)+\phi-1}}.
\end{equation}

It is woth mentioning that $(\phi,\mu,\alpha)$-reference prior is the same as the $(\phi,\alpha,\mu)$-reference prior, while $(\mu,\alpha,\phi)$-reference prior has the same form of $\pi_8(\boldsymbol{\theta})$ which completes all possible reference priors obtained from Proposition \ref{propositionop}.


\begin{corollary}\label{maintheproper}  The prior $\pi_{10}\left(\phi\right)$ has the asymptotic power-law behavior given by 
\begin{equation*}
\pi_{10}\left(\phi\right)\underset{\phi\to 0^{+}}{\propto} \phi^{-\frac{1}{2}} \quad \mbox{ and } \quad \pi_{10}\left(\phi\right)\underset{\phi\to \infty}{\propto} \phi^{-\frac{3}{2}}, 
\end{equation*}
then the obtained posterior distribution is proper for $n\geq 2$ and its higher moments are improper for all $n\in\N^+$.
\end{corollary}
\begin{proof} From (\ref{equatinpropi}) and by the asymptotic relations (\ref{digammatoinfty}) we have that 
\begin{equation*}
\phi^2\psi^{'}(\phi)+\phi-1 =2\phi-\frac{1}{2}+o\left(1\right) \underset{\phi\to \infty}{\propto} \phi
\end{equation*}
which together with equation (\ref{eqaux2}) implies that
\begin{equation*}
\sqrt{\phi^2\psi^{'}(\phi)+\phi-1}\underset{\phi\to 0^{+}}{\propto}  \sqrt{\phi}\ \mbox{ and }\ \sqrt{\phi^2\psi^{'}(\phi)+\phi-1}\underset{\phi\to \infty}{\propto} \sqrt{\phi}.
\end{equation*}
Hence, from the above proportionalities we have that
\begin{equation*}
\sqrt{\frac{\phi^2\psi'(\phi)^2-\psi'(\phi)-1}{\phi^2\psi'(\phi)+\phi-1}} \underset{\phi\to 0^{+}}{\propto}\phi^{-\frac{1}{2}} \ \ \mbox{ and } \ \ \sqrt{\frac{\phi^2\psi'(\phi)^2-\psi'(\phi)-1}{\phi^2\psi'(\phi)+\phi-1}} \underset{\phi\to \infty}{\propto} \phi^{-\frac{3}{2}}.
\end{equation*}
Therefore, Theorem \ref{fundteo1} can be applied with $k= q_0 = q_\infty = -1$, $r_0 = -\frac{1}{2}$ and $r_\infty= -\frac{3}{2}$ where $k=-1$, $q_\infty < r_0$ and $2r_\infty+1 < q_0$, and therefore $\pi_{10}(\alpha,\mu,\phi)$ leads to a proper posterior for every $n> -q_0 = 1$.

In order to prove that the higher moments are improper suppose  $\alpha^q\phi^r\mu^j\pi(\boldsymbol{\theta})$ leads to a proper posterior for $r\in\N$, $q\in\N$ and $k\in\N$. By Theorem \ref{fundteo3} we have $j= 0$, $q+q_\infty<r+r_0$, $2(r+r_\infty) \leq q+q_0$ and $n\geq -q_0$, i.e., $k=0$ and $2r-1< q< r+\frac{1}{2}$. The inequality $2r-1<r+\frac{1}{2}$ leads to $r<\frac{3}{2}$, i.e., $r=0$ or $r=1$. By the previous inequality, the case where $r=0$ leads to $-1<q<\frac{1}{2}$, that is, $q=0$. Now, for $r=1$ we have the inequality $1<q<\frac{3}{2}$ which do not have integer solution. Therefore, the only possible values for which $\alpha^q\phi^r\mu^j\pi(\boldsymbol{\theta})$ is proper is $q=r=j=0$, that is, the higher moments are improper. 
\end{proof}

\section{Discussion}
 
Objective priors play an important role in Bayesian analysis. For several important distributions, we showed that such objective priors are improper prior and may lead to improper posterior; in these cases, the Bayesian inference cannot be conducted, which is undesirable. An exciting aspect of our findings is that such priors either follows a power-law distribution or present an asymptotic behavior to this distribution. Our mathematical formalism is general and covers important distributions widely used in the literature. The exponent of the obtained power-law distributions is contained between 0.5 and 1. Hence they are improper with infinite mean and variance. 

We provided sufficient and necessary conditions for the posteriors to be proper, depending on the exponent of the power-law model. For instance, if $\phi$ is known the $(\alpha,\mu)$-reference prior for the Weibull and Generalized half-normal distributions, the priors follow power-law distributions with exponent one and returned proper posteriors. By considering $\alpha$ fixed, we showed that both the  Jeffreys' first rule and the Jeffreys' prior returned proper posterior distributions as well as finite higher moments, which are valid for the Gamma, Nakagami-m and Wilson-Hilferty distributions. Moreover, we provided many situations were the obtained posterior are improper and should not be used, opening new opportunities for the analysis of real data.

The observed behavior also occurs in many other classes of distributions, for instance, for the Lomax distribution, which is a modified version of the Pareto model, the reference prior for the two parameters of the model follows power-law distributions with exponent one \cite{ferreira2020objective}. This behavior is also observed in a Gaussian distribution when $\mu$ is a known parameter, in this case, the Jeffreys prior for standard deviation $\sigma$ follows a power-law distribution with exponent one and the obtained posterior is proper. Under the Behrens-Fisher problem, the obtained Jeffreys prior for the parameters have the same behavior with exponents two while the reference prior has exponents three \cite{liseo}. There are a large number of possible extensions of this current work. The power-law distributions may be used as objective prior in the models when there is the presence of censored data or long-term survival; in these cases, it is difficult or impossible to obtain such objective priors.  The study of the behavior for other distributions, such as generalized linear models, should also be further investigated.

\section*{Disclosure statement}

No potential conflict of interest was reported by the author(s)

\section*{Acknowledgements}

Pedro L. Ramos acknowledges support from the S\~ao Paulo State Research Foundation (FAPESP Proc. 2017/25971-0).

\bibliographystyle{tfs}

\bibliography{reference}

 \section*{Appendix A:}    
\subsection{Useful Proportionalities}

The following proportionalities are useful to prove results related to the posterior distribution, and its proofs can be seen in \cite{ramos2017bayesian}.


\begin{proposition}\label{lim2b} 
Let $\f{p}(\alpha)=\log\left(\dfrac{\frac{1}{n}\sum_{i=1}^n t_i^\alpha}{{\sqrt[n]{\prod_{i=1}^n t_i^\alpha}}}\right)$, $\f{q}(\alpha)=\f{p}(\alpha)+\log{n}$,  for $t_1,t_2,\ldots,t_n$ positive and not all equal, $h\in\mathbb{R}^+$, $r\in\mathbb{R}^+$ and  $t_m = \max\{t_1,\ldots,t_n\}$, then $\f{p}(\alpha)>0$, $\f{q}(\alpha)>0$ and the following results hold
\begin{equation*}
\f{p}(\alpha)\underset{\alpha\to 0^{+}}{\propto} \alpha^2 \quad \mbox{and} \quad \f{p}(\alpha)\underset{\alpha\to \infty}{\propto} \alpha;
\end{equation*}
\begin{equation*}
\f{q}(\alpha)\underset{\alpha\to 0^{+}}{\propto} 1 \quad \mbox{and} \quad \f{q}(\alpha)\underset{\alpha\to \infty}{\propto} \alpha;
\end{equation*}
\begin{equation*}
 \dfrac{\Gamma(n\phi)}{\Gamma(\phi)^n}\underset{\phi\to 0^+}{\propto} \phi^{n-1} 
\ \ \mbox{ and } \ \ \dfrac{\Gamma(n\phi)}{\Gamma(\phi)^n}\underset{\phi\to \infty}{\propto} \phi^{\frac{n-1}{2}}n^{n\phi};
\end{equation*}
\begin{equation}
\gamma\left(h,r\f{q}(\alpha)\right) \underset{\alpha\to 0^{+}}{\propto} 1 \quad \mbox{and } \quad \gamma\left(h,r\f{q}(\alpha)\right) \underset{\alpha\to \infty}{\propto} 1;
\end{equation}
\begin{equation}
\Gamma\left(h,r\f{p}(\alpha)\right) \underset{\alpha\to 0^{+}}{\propto} 1 \quad \mbox{and } \quad \Gamma\left(h,r\f{p}(\alpha)\right) \underset{\alpha\to \infty}{\propto} \alpha^{k-1}e^{-r k(\boldsymbol{x}) \alpha};
\end{equation}
where $k(\boldsymbol{x})=\log\left(\frac{t_m}{\sqrt[n]{\prod_{i=1}^n t_i}}\right)>0$; $\gamma(y,x)=1-\Gamma(y,x)$ and $\Gamma(y,x)=\int_{x}^{\infty}{w^{y-1}e^{-w}}\, dw$  is the upper incomplete gamma function.
\end{proposition}

\subsection{Proof of Proposition \ref{proportional1}}\label{propsitionaprof}

Suppose that $\f{g}(x)\underset{x\to a}{\lesssim} \f{h}(x)$ and $\f{g}(x)\underset{x\to b}{\lesssim} \f{h}(x)$. Then, by Definition \ref{definition1} we have that $\limsup_{x\to a} \dfrac{\f{g}(x)}{\f{h}(x)} = w$ for some $w\in \R^+$. Therefore, from the  definition of $\limsup$ there exist some $a'\in(a,b)$ such that $\dfrac{\f{g}(x)}{\f{h}(x)} \leq \dfrac{3w}{2}$ for every $x\in (a,a']$.
Proceeding analogously, there must exist some $v
\in \R^+$ and $b'\in (a',b)$ such that $\dfrac{\f{g}(x)}{\f{h}(x)} \leq \dfrac{3v}{2}$ for every $x\in [b',b)$.
On the other hand, since $\dfrac{\f{g}(x)}{\f{h}(x)}$ is continuous in $[a',b']$, the Weierstrass Extreme Value Theorem states that there exist some $x_1 \in[a',b']$ such that $\dfrac{\f{g}(x)}{\f{h}(x)} \leq \dfrac{\f{g}(x_1)}{\f{h}(x_1)}$ for every $x\in [a',b']$. Finally, choosing $M=\max\left(\dfrac{3w}{2},\dfrac{3v}{2},\dfrac{\f{g}(x_1)}{\f{h}(x_1)}\right)<\infty$, it follows that $\dfrac{\f{g}(x)}{\f{h}(x)} \leq M$ for every $x\in (a,b)$, which by Definition \ref{definition0} means that $g(x)\lesssim h(x)$.

Now suppose $g(x)\lesssim h(x)$. By Definition \ref{definition0}, there exist some $M<0$ such that $\dfrac{\f{g}(x)}{\f{h}(x)} \leq M$ for every $x\in(a,b)$. This implies that $\limsup_{x\to a} \frac{\f{g}(x)}{\f{h}(x)} \leq M < \infty$ which by Definition \ref{definition1} means that $\f{g}(x)\underset{x\to a}{\lesssim} \f{h}(x)$. The proof that $\f{g}(x)\underset{x\to b}{\lesssim} \f{h}(x)$ must also be satisfied is analogous to the previous case. Therefore the theorem is proved.

\subsection{Proof of Theorem \ref{fundteo1alpha}}\label{ctheoremaalph1}

Let $\alpha\in \R^+$ be fixed. Since $\frac{\pi(\phi)}{\Gamma(\phi)^n}\left\{\prod_{i=1}^n{x_i^{\alpha\phi}}\right\}\pi(\mu)\mu^{n\alpha\phi-1}\exp{\left\{-\mu^{\alpha}\sum_{i=1}^n x_i^\alpha\right\}}\geq0$ always, by Tonelli's theorem we have:
\begin{equation*}
\begin{aligned}
d(\boldsymbol{x};\alpha) & = \int\limits_{\mathcal{A}}\frac{\pi(\phi)}{\Gamma(\phi)^n}\left\{\prod_{i=1}^n{x_i^{\alpha\phi-1}}\right\}\pi(\mu)\mu^{n\alpha\phi}\exp\left\{-\mu^{\alpha}\sum_{i=1}^n x_i^\alpha\right\}d\boldsymbol{\theta} \\
& =  \int\limits_0^\infty \int\limits_0^\infty \frac{\pi(\phi)}{\Gamma(\phi)^n}\left\{\prod_{i=1}^n{x_i^{\alpha\phi-1}}\right\}\pi(\mu)\mu^{n\alpha\phi}\exp\left\{-\mu^{\alpha}\sum_{i=1}^n x_i^\alpha\right\}\, d\mu\, d\phi.
\end{aligned}
\end{equation*}

Since $\pi(\mu) \lesssim \mu^{k}$ and $k\geq -1$ by hypothesis it follows that
\begin{equation*}
\begin{aligned}
d(\boldsymbol{x};\alpha) &\lesssim    
\int\limits_0^\infty\int\limits_0^\infty\frac{\pi(\phi)}{\Gamma(\phi)^n}\left(\prod_{i=1}^n{x_i^{\alpha}}\right)^{\phi}\mu^{n\alpha\phi+k}\exp\left\{-\mu^{\alpha}\sum_{i=1}^n x_i^\alpha\right\} \, d\mu \, d\phi \\&
=\int\limits_0^\infty\int\limits_0^{\infty}\frac{\pi(\phi)}{\Gamma(\phi)^n}\left(\prod_{i=1}^n{x_i^{\alpha}}\right)^{\phi}\frac{\alpha\Gamma\left(n\phi+\frac{k+1}{\alpha} \right)}{\left(\sum_{i=1}^n x_i^\alpha\right)^{n\phi+\frac{k+1}{\alpha} }}\, d\mu \, d\phi .
\end{aligned}
\end{equation*}

Now suppose that $k>-1$. Then, since $k+1>0$, $\Gamma(n\phi + \frac{k+1}{\alpha})\underset{\phi\to 0^+}{\propto}1$ and $\Gamma(n\phi +\frac{k+1}{\alpha})\underset{\phi\to \infty}{\propto} \Gamma(n\phi)(n\phi)^{\frac{k+1}{\alpha}}$ (see \cite{abramowitz}). Therefore, from the proportionalities in Proposition \ref{lim2b} it follows that
\begin{equation}\label{eqmain23alpha}
\begin{aligned}
d(\boldsymbol{x};\alpha)
& \lesssim \int\limits_0^{1}\pi(\phi)\frac{1}{\Gamma(\phi)^n}e^{-n\f{q}(\alpha)\phi} d\phi\, + \int\limits_1^{\infty}\pi(\phi)\frac{\Gamma(n\phi)}{\Gamma(\phi)^n}\phi^{\frac{k+1}{\alpha}}e^{-n\f{q}(\alpha)\phi}\, d\phi\\
&\propto  \int\limits_0^{1}\pi(\phi)\phi^{n}e^{-n\f{q}(\alpha)\phi} d\phi\, + \int\limits_1^{\infty}\pi(\phi)\phi^{\frac{n-1}{2}+\frac{k+1}{\alpha}}e^{-n\f{p}(\alpha)\phi}\, d\phi =s_1(\boldsymbol{x};\alpha)+s_2(\boldsymbol{x};\alpha)
\end{aligned}
\end{equation}
where $\f{q(\alpha)}$ and $\f{p(\alpha)}$ are given in Proposition \ref{lim2b} and $s_1(\boldsymbol{x};\alpha)$ and $s_2(\boldsymbol{x};\alpha)$ denote the respective two integrals in the sum that precedes it. It follows that $d(\boldsymbol{x};\alpha) < \infty$ if $s_1(\boldsymbol{x};\alpha)<\infty$ and $s_2(\boldsymbol{x};\alpha)<\infty$. Now, using the proportionalities in Proposition \ref{lim2b} it follows that, since $n + r_0 > -1$, $\f{q}(\alpha)>0$ and $\f{p}(\alpha)>0$, then
\begin{equation*}
\begin{aligned}
s_1(\boldsymbol{x};\alpha)
& \lesssim\int\limits_0^{1}\phi^{n+r_0}e^{-n\f{q}(\alpha)\phi} \, d\phi \,  =  \frac{\gamma(n+r_0+1,n\f{q}(\alpha))}{(n\f{q}(\alpha))^{n+r_0}}\,
 < \infty,
\end{aligned}
\end{equation*}
and
\begin{equation*}
\begin{aligned}
s_2(\boldsymbol{x};\alpha)
 &\lesssim \int\limits_1^\infty \phi^{\frac{n+1+2r_\infty}{2}+\frac{k+1}{\alpha}-1}e^{-n\f{p}(\alpha)\phi}\, d\phi\,
 =  \frac{\Gamma(\frac{n+1+2r_\infty}{2}+\frac{k+1}{\alpha},n\f{p}(\alpha))}{(n\f{p}(\alpha))^{\frac{n+1+2r_\infty}{2}+\frac{k+1}{\alpha}}}\, < \infty,
\end{aligned}
\end{equation*}
therefore, we have that $d(\boldsymbol{x};\alpha)<\infty$.

The case where $k=-1$ and $n>-r_0$ is completely analogous to the previous case, with the only difference in the proof being that $\Gamma(n\phi + \frac{k+1}{\alpha})\underset{\phi\to 0^+}{\propto}\phi^{-1}$ in this case, instead of $\Gamma(n\phi + \frac{k+1}{\alpha})\underset{\phi\to 0^+}{\propto}1$.

\subsection{Proof of Theorem \ref{fundteo2alpha}}\label{ctheoremaalph2}

Let $\alpha\in \R^+$ be fixed. Suppose that hypothesis of item $i)$ hold, that is, $\pi(\mu) \gtrsim \mu^k$ with $k < -1$. Notice that, for $0<\phi\leq-\frac{(k+1)}{n\alpha}$ we have that $n\alpha\phi+k\leq -1$. Moreover, for every $\alpha>0$ fixed we have that $\exp\left\{-\mu^{\alpha}\sum_{i=1}^n x_i^\alpha\right\}\underset{\mu\to 0^{+}}{\propto} 1$. Hence, from Proposition \ref{proposition1} we have that
\begin{equation*}
\begin{aligned}
\int\limits_0^\infty \pi(\mu)\mu^{n\alpha\phi}\exp{\left\{-\mu^{\alpha}\sum_{i=1}^n x_i^\alpha\right\}} \, d\mu
\gtrsim \int\limits_0^1 \mu^{n\alpha\phi+k} d\mu = \infty,
\end{aligned}
\end{equation*}
for all $\phi\in (0,-\frac{(k+1)}{n\alpha}]$. Therefore
\begin{equation*}
\begin{aligned}
d(\boldsymbol{x};\alpha) & \gtrsim \int\limits_{0}^{-\frac{(k+1)}{n\alpha}} \frac{\pi(\phi)}{\Gamma(\phi)^n}\left(\prod_{i=1}^n{x_i^{\alpha}}\right)^{\phi}\int\limits_0^{\infty}\mu^{n\alpha\phi+k}\exp{\left\{-\mu^{\alpha}\sum_{i=1}^n x_i^\alpha\right\}} \, d\mu \, d\phi \\
& \gtrsim \int\limits_{0}^{-\frac{(k+1)}{n\alpha}} \infty\; d\phi\;  = \infty,
\end{aligned}
\end{equation*}
that is, $d(\boldsymbol{x};\alpha)=\infty$. 

Now suppose that hypothesis of $ii)$ hold. First suppose that $\pi(\mu) \underset{\mu\to \infty}{\gtrsim} \mu^{k}$ and $\pi(\phi) \underset{\phi\to 0^+}{\gtrsim} \phi^{r_0}$, where $k > -1$ and $n<-r_0-1$. Then, following the same steps that resulted in (\ref{eqmain23alpha})  we have that
\begin{equation*}
\begin{aligned}
d(\boldsymbol{x};\alpha) & \gtrsim \int\limits_0^{1}\phi^{n+r_0}e^{-n\f{q}(\alpha)\phi} \, d\phi \propto \int\limits_0^{1}\phi^{n+r_0} \, d\phi =\infty
\end{aligned}
\end{equation*}
and therefore $d(\boldsymbol{x};\alpha)=\infty$.

The case where $k=-1$, and $n<-r_0$ follows analogously

\subsection{Proof of Theorem \ref{fundteophi1}}\label{ctheoremphi1}

Let $\phi\in \R^+$ be fixed. Since $\pi(\alpha)\alpha^{n}\frac{\pi(\phi)}{\Gamma(\phi)^n}\left\{\prod_{i=1}^n{x_i^{\alpha\phi}}\right\}\pi(\mu)\mu^{n\alpha\phi-1}\exp{\left\{-\mu^{\alpha}\sum_{i=1}^n x_i^\alpha\right\}}\geq0$ always, by Tonelli's theorem we have:
\begin{equation}\label{postdemapb2}
\begin{aligned}
d(\boldsymbol{x};\phi) & = \int\limits_{\mathcal{A}}\pi(\alpha)\alpha^{n}\left\{\prod_{i=1}^n{x_i^{\alpha\phi-1}}\right\}\pi(\mu)\mu^{n\alpha\phi}\exp\left\{-\mu^{\alpha}\sum_{i=1}^n x_i^\alpha\right\}d\boldsymbol{\theta} \\
& = \int\limits_0^\infty  \int\limits_0^\infty \pi(\alpha)\alpha^{n}\left\{\prod_{i=1}^n{x_i^{\alpha\phi-1}}\right\}\pi(\mu)\mu^{n\alpha\phi}\exp\left\{-\mu^{\alpha}\sum_{i=1}^n x_i^\alpha\right\}\, d\mu\, d\alpha.
\end{aligned}
\end{equation}

Now, since $\pi(\mu) \lesssim \mu^{-1}$ by hypothesis it follows that
\begin{equation*}
\begin{aligned}
d(\boldsymbol{x};\phi) & \lesssim  \int\limits_0^\infty\int\limits_0^{\infty}\pi(\alpha)\alpha^{n}\left(\prod_{i=1}^n{x_i^{\alpha}}\right)^{\phi}\mu^{n\alpha\phi-1}\exp{\left\{-\mu^{\alpha}\sum_{i=1}^n x_i^\alpha\right\}} \, d\mu  \, d\alpha \\
&=\int\limits_0^\infty\pi(\alpha)\alpha^{n-1}\dfrac{\left(\prod_{i=1}^n{x_i^{\alpha}}\right)^\phi}{\left(\sum_{i=1}^n x_i^\alpha\right)^{n\phi}}\, d\alpha =\int\limits_0^{\infty}\pi(\alpha)\alpha^{n-1}e^{-n\f{q}(\alpha)\phi}\, d\alpha
\end{aligned}
\end{equation*}
where $\f{q}(\alpha)$ is given in Proposition \ref{lim2b}. Therefore, from the proportionalities in Proposition \ref{lim2b} it follows that
\begin{equation}\label{eqmain23weibull}
\begin{aligned}
d(\boldsymbol{x};\phi)
& \lesssim \int\limits_0^{\infty}\pi(\alpha)\alpha^{n-1}e^{-n\f{q}(\alpha)\phi}\, d\alpha
\\ &\propto\int\limits_0^{1}\alpha^{q_0+n-1}e^{-n\f{q}(\alpha)\phi}\, d\alpha + \int\limits_1^{\infty}\alpha^{q_\infty+n-1}e^{-n\f{q}(\alpha)\phi}\, d\alpha=s_1(\boldsymbol{x};\phi)+s_2(\boldsymbol{x};\phi).
\end{aligned}
\end{equation}
where $s_1(\boldsymbol{x};\phi)$ and $s_2(\boldsymbol{x};\phi)$ denote the respective two real numbers in the sum that precedes it. It follows that $d(\boldsymbol{x};\phi) < \infty$ if $s_1(\boldsymbol{x};\phi)<\infty$ and $s_2(\boldsymbol{x};\phi)<\infty$. 

By Proposition \ref{lim2b}, $\f{q}(\alpha)>0$, which implies that $e^{-n\f{q}(\alpha)\phi}\leq 1$. Moreover, since $q_0+ n > 0$ we have that
\begin{equation*}
\begin{aligned}
s_1(\boldsymbol{x};\phi)
& = \int\limits_0^{1}\alpha^{q_0+n-1}e^{-n\f{q}(\alpha)\phi}\, d\alpha & \leq \int\limits_0^{1}\alpha^{q_0+n-1}\, d\alpha < \infty
\end{aligned}
\end{equation*}

Additionally, by Proposition \ref{lim2b}, $\f{q}(\alpha)\underset{\alpha\to \infty}{\propto} \alpha$ and therefore by Proposition \ref{proportional1} there exists $c>0$ such that $\f{q}(\alpha) \leq c\alpha$ for all $\alpha\in [1,\infty)$. Therefore

\begin{equation*}
\begin{aligned}
s_2(\boldsymbol{x};\phi)
= \int\limits_1^{\infty}\alpha^{q_\infty+n-1}e^{-n\f{q}(\alpha)\phi}\, d\alpha 
\leq \int\limits_1^{\infty} \alpha^{q_\infty+n-1}e^{-n\phi c\alpha}\, d\alpha
= \frac{\Gamma(q_\infty+n,n\phi c)}{(n\phi c)^{q_\infty+n}}< \infty,
\end{aligned}
\end{equation*}
hence, $d(\boldsymbol{x};\phi)<\infty$.

\subsection{Proof of Theorem \ref{fundteophi2}}\label{ctheoremphi2}

Let $\phi\in \R^+$ be fixed. Suppose that $\pi(\mu)  \gtrsim \mu^{k}$ where  $k< -1$. Notice that,  for $0<\alpha\leq\frac{k+1}{n\phi}$ it follows that $n\phi+\frac{k+1}{\alpha}\leq0$  and since $\exp\left\{-\mu^{\alpha}\sum_{i=1}^n x_i^\alpha\right\}\underset{\mu\to 0^{+}}{\propto} 1$, we have that
\begin{equation*}
\begin{aligned}
\int\limits_0^\infty \pi(\mu)\mu^{n\alpha\phi}\exp{\left\{-\mu^{\alpha}\sum_{i=1}^n x_i^\alpha\right\}} \, d\mu
\gtrsim \int\limits_0^1 \mu^{n\alpha\phi+k} d\mu = \infty,
\end{aligned}
\end{equation*}
for all $\alpha\in (0,\frac{k+1}{n\phi}]$. Therefore 
\begin{align*}
d(\boldsymbol{x};\phi) \gtrsim \int\limits_0^{\frac{k+1}{n\phi}}\pi(\alpha)\alpha^{n-1}\left(\prod_{i=1}^n{x_i^{\alpha}}\right)^\phi\int\limits_{0}^1 \pi(\mu)\mu^{n\alpha\phi}\exp{\left\{-\mu^{\alpha}\sum_{i=1}^n x_i^\alpha\right\}}\, d\mu\, d\alpha     = \int_0^{\frac{k+1}{n\phi}} \infty \; d\alpha= \infty
\end{align*}
hence $d(\boldsymbol{x};\phi)=\infty$. 

 Now suppose that
 $\pi(\mu) \gtrsim \mu^{k}$ and $\pi(\alpha) \underset{\alpha\to 0^+}{\gtrsim} \alpha^{q_0}$, where $k > -1$ and $q_0\in\R$. Then
\begin{equation*}
\begin{aligned}
d(\boldsymbol{x};\phi) &  \gtrsim \int\limits_0^1\int\limits_0^{\infty}\alpha^{n+q_0}\left(\prod_{i=1}^n{x_i^{\alpha}}\right)^{\phi}\mu^{n\alpha\phi+k}\exp\left\{-\mu^{\alpha}\sum_{i=1}^n x_i^\alpha\right\} \, d\mu\, d\alpha \\
&= \int\limits_0^1\int\limits_0^\infty\alpha^{n+q_0}\frac{\left(\prod_{i=1}^n{x_i^{\alpha}}\right)^\phi}{\left(\sum_{i=1}^n x_i^\alpha\right)^{n\phi+\frac{k+1}{\alpha}}}u^{n\phi+\frac{k+1}{\alpha}-1}e^{-u}\, du\, d\alpha =\\
&= \int\limits_0^1\int\limits_0^\infty\alpha^{n+q_0}\left(\prod_{i=1}^n{x_i^\alpha}\right)^{-\frac{k+1}{\alpha}}n^{-n\phi -\frac{k+1}{\alpha}}e^{-\f{p}(\alpha)\left(n\phi+\frac{k+1}{\alpha}\right)}u^{n\phi+\frac{k+1}{\alpha}-1}e^{-u}\, du\, d\alpha\\
&= \int\limits_0^{\infty}\left(\prod_{i=1}^n{x_i}\right)^{-(k+1)}n^{-n\phi}u^{n\phi-1}e^{-u}\int\limits_0^1\alpha^{n+q_0}e^{-\f{p}(\alpha)\left(n\phi+\frac{k+1}{\alpha}\right)}e^{(\log{u}-\log{n}) \frac{k+1}{\alpha}}\, d\alpha\, du
\end{aligned}
\end{equation*}
where in the above we used the change of variables $u = \mu^\alpha \sum_{i=1}^n x_i^\alpha$ in the integral and $\f{p}(\alpha)$ is given as in Proposition \ref{lim2b}.

Now, since $\f{p}(\alpha)\underset{\alpha\to 0^{+}}{\propto} \alpha^2$ from Proposition \ref{lim2b} it follows that $\lim_{\alpha\to 0^+} e^{-\f{p}(\alpha)\left(n\phi+\frac{k+1}{\alpha}\right)} = \lim_{\alpha\to 0^+} e^{-\frac{p(\alpha)}{\alpha^2}\left(n\phi \alpha + k + 1\right)\alpha} = e^0 = 1$. These two facts together applied to the above inequality leads to
\begin{equation*}
\begin{aligned}
d(\boldsymbol{x};\phi) 
&\gtrsim \int\limits_0^{\infty}n^{-n\phi}\left(\prod_{i=1}^n{x_i}\right)^{-(k+1)}u^{n\phi-1}e^{-u}\int\limits_0^1\alpha^{n+q_0}e^{(\log{u}-\log{n}) \frac{k+1}{\alpha}}\, d\alpha\, du
\end{aligned}
\end{equation*}

Thus, since $n\geq 1$ and $\log u - \log n > 0$ for $u\geq 3n > e\cdot n$, and since $\int_0^1 \alpha^H e^{\frac{L}{\alpha}} = \infty$ for every $H\in \R$ and $L\in R^+$ (which can be easily checked via the change of variable $\beta = \frac{1}{\alpha}$ in the integral), it follows that
\begin{equation}\label{weibullpos}
\begin{aligned}
d(\boldsymbol{x};\phi) &\gtrsim\int\limits_0^{\infty}n^{-n\phi}\left(\prod_{i=1}^n{x_i}\right)^{-(k+1)}u^{n\phi-1}e^{-u}\cdot \infty\, du\, = \infty,
\end{aligned}
\end{equation}
and therefore $d(\boldsymbol{x};\phi) = \infty$.





Now suppose that
 $\pi(\mu) \underset{\mu\to \infty}{\gtrsim} \mu^{k}$ and $\pi(\alpha) \underset{\alpha\to 0^+}{\gtrsim} \alpha^{q_0}$, where $k \leq -1$ and
 $n\leq -q_0$. Then, following the same steps that resulted in (\ref{eqmain23weibull}) we have that
\begin{equation*}
\begin{aligned}
d(\boldsymbol{x};\phi)
&\gtrsim \int\limits_0^{1}\alpha^{q_0+n-1}e^{-n\f{q}(\alpha)\phi}\, d\alpha.
\end{aligned}
\end{equation*}
but since by Proposition \ref{lim2b} we have that $q(\alpha)\underset{\alpha\to 0^+}{\propto} 0$ it follows that $e^{-nq(\alpha)\phi}\underset{\alpha\to 0^+}{\propto} 1$ and therefore
\begin{equation*}
\begin{aligned}
d(\boldsymbol{x};\phi) & \gtrsim \int\limits_0^{1}\alpha^{q_0+n-1}\, d\alpha = \infty.
\end{aligned}
\end{equation*}

\subsection{Proof of Theorem \ref{fundteo1}}\label{ctheoremab2}

Since $\pi(\alpha)\alpha^{n}\frac{\pi(\phi)}{\Gamma(\phi)^n}\left\{\prod_{i=1}^n{x_i^{\alpha\phi}}\right\}\pi(\mu)\mu^{n\alpha\phi-1}\exp{\left\{-\mu^{\alpha}\sum_{i=1}^n x_i^\alpha\right\}}\geq0$ always, by Tonelli's theorem we have:
\begin{equation*}
\begin{aligned}
d(\boldsymbol{x}) & = \int\limits_{\mathcal{A}}\pi(\alpha)\alpha^{n}\frac{\pi(\phi)}{\Gamma(\phi)^n}\left\{\prod_{i=1}^n{x_i^{\alpha\phi-1}}\right\}\pi(\mu)\mu^{n\alpha\phi}\exp\left\{-\mu^{\alpha}\sum_{i=1}^n x_i^\alpha\right\}d\boldsymbol{\theta} \\
& = \int\limits_0^\infty \int\limits_0^\infty \int\limits_0^\infty \pi(\alpha)\alpha^{n}\frac{\pi(\phi)}{\Gamma(\phi)^n}\left\{\prod_{i=1}^n{x_i^{\alpha\phi-1}}\right\}\pi(\mu)\mu^{n\alpha\phi}\exp\left\{-\mu^{\alpha}\sum_{i=1}^n x_i^\alpha\right\}\, d\mu\, d\phi\, d\alpha.
\end{aligned}
\end{equation*}

Now, since $\pi(\mu) \lesssim \mu^{-1}$ we have that \\
\begin{equation*}
\begin{aligned}
d(\boldsymbol{x}) & \lesssim  \int\limits_0^\infty\int\limits_0^\infty\int\limits_0^{\infty}\pi(\alpha)\alpha^{n}\frac{\pi(\phi)}{\Gamma(\phi)^n}\left(\prod_{i=1}^n{x_i^{\alpha}}\right)^{\phi}\mu^{n\alpha\phi-1}\exp{\left\{-\mu^{\alpha}\sum_{i=1}^n x_i^\alpha\right\}} \, d\mu \, d\phi \, d\alpha \\
&=\int\limits_0^\infty\int\limits_0^\infty\pi(\alpha)\alpha^{n-1}\frac{\pi(\phi)}{\Gamma(\phi)^n}\left(\prod_{i=1}^n{x_i^{\alpha}}\right)^\phi\dfrac{\Gamma\left(n\phi\right)}{\left(\sum_{i=1}^n x_i^\alpha\right)^{n\phi}}\, d\phi\, d\alpha \\
&=\int\limits_0^{\infty}\int\limits_0^{\infty}\pi(\alpha)\alpha^{n-1}\pi(\phi)\frac{\Gamma\left(n\phi\right)}{\Gamma(\phi)^n}e^{-n\f{q}(\alpha)\phi}\, d\phi\, d\alpha
\end{aligned}
\end{equation*}
where $\f{q}(\alpha)$ is given in Proposition \ref{lim2b}. Therefore, from the proportionalities in Proposition \ref{lim2b} it follows that
\begin{equation}\label{eqmain23}
\begin{aligned}
d(\boldsymbol{x})
& \lesssim \int\limits_0^{\infty}\int\limits_0^{\infty}\pi(\alpha)\alpha^{n-1}\pi(\phi)\frac{\Gamma(n\phi)}{\Gamma(\phi)^n}e^{-n\f{q}(\alpha)\phi}\, d\phi\, d\alpha
\\ &\propto \int\limits_0^{1}\int\limits_0^{1}f(\alpha,\phi)\, d\phi\, d\alpha + \int\limits_1^{\infty}\int\limits_0^{1}f(\alpha,\phi)\, d\phi \, d\alpha + \int\limits_0^{1}\int\limits_1^{\infty}g(\alpha,\phi)\, d\phi\, d\alpha + \int\limits_1^{\infty}\int\limits_1^{\infty}g(\alpha,\phi)\, d\phi\, d\alpha \\&=s_1(\boldsymbol{x})+s_2(\boldsymbol{x})+s_3(\boldsymbol{x})+s_4(\boldsymbol{x}),
\end{aligned}
\end{equation}
where $f(\alpha,\phi)=\pi(\alpha)\alpha^{n-1} \pi(\phi)\phi^{n-1}e^{-n\f{q}(\alpha)\phi}$, $g(\alpha,\phi)=\pi(\alpha)\alpha^{n-1} \pi(\phi)\phi^{\frac{n-1}{2}}e^{-n\f{p}(\alpha)\phi}$ and $s_1(\boldsymbol{x})$, $s_2(\boldsymbol{x})$, $s_3(\boldsymbol{x})$ and $s_4(\boldsymbol{x})$ denote the respective four real numbers in the sum that precedes it. It follows that $d(\boldsymbol{x}) < \infty$,  if and only if $s_1(\boldsymbol{x})<\infty$, $s_2(\boldsymbol{x})<\infty$, $s_3(\boldsymbol{x})<\infty$ and $s_4(\boldsymbol{x})<\infty$. Now, using the proportionalities in Proposition \ref{lim2b} it follows that

\begin{equation*}
\begin{aligned}
s_1(\boldsymbol{x})
& \lesssim \int\limits_0^{1} \alpha^{q_0+n-1}\int\limits_0^{1}\phi^{n+r_0-1}e^{-n\f{q}(\alpha)\phi} \, d\phi \, d\alpha \\
&= \int\limits_0^{1} \alpha^{q_0+n-1}\frac{\gamma(n+r_0,n\f{q}(\alpha))}{(n\f{q}(\alpha))^{n+r_0}}\, d\alpha
\propto \int\limits_0^{1} \alpha^{q_0+n-1}\, d\alpha < \infty,
\end{aligned}
\end{equation*}
where in the last inequality the condition $n>-q_0$ was used, and in the equality that precedes it the condition $n>-r_0$ was used to ensure that $\gamma(n+r_0,n\f{q}(\alpha))$ is well defined and that the equality holds,
\begin{equation*}
\begin{aligned}
s_2(\boldsymbol{x})
&\lesssim \int\limits_1^{\infty}\alpha^{q_\infty+n-1}\int\limits_0^{1}\phi^{n+r_0-1}e^{-n\f{q}(\alpha)\phi} \, d\phi \, d\alpha \\
& = \int\limits_1^{\infty} \alpha^{q_\infty+n-1}\frac{\gamma(n+r_0,n\f{q}(\alpha))}{(n\f{q}(\alpha))^{n+r_0}}\, d\alpha
\propto \int\limits_1^{\infty}\alpha^{q_\infty-r_0-1}\, d\alpha < \infty,
\end{aligned}
\end{equation*}
where just as in the $s_1(\boldsymbol{x})$ case, the condition $n>-r_0$ was used in order for the above equality to hold,
\begin{equation*}\label{demanes32}
\begin{aligned}
s_3(\boldsymbol{x})
&\lesssim \int\limits_0^1 \alpha^{q_0+n-1} \int\limits_1^\infty \phi^{\tfrac{n+1+2r_\infty}{2}-1}e^{-n\f{p}(\alpha)\phi}\, d\phi\, d\alpha\\
& = \int\limits_0^1 \alpha^{q_0+n-1}\frac{\Gamma(\frac{n+1+2r_\infty}{2},n\f{p}(\alpha))}{(n\f{p}(\alpha))^{\frac{n+1+2r_\infty}{2}}}\, d\alpha
\propto \int\limits_0^1\alpha^{q_0-2r_\infty-2}\, d\alpha < \infty,
\end{aligned}
\end{equation*}
where in the last inequality the condition $q_0>2r_\infty+1$ was used, and finally
\begin{equation*}\label{demanes34}
\begin{aligned}
s_4(\boldsymbol{x})
&\lesssim \int\limits_1^{\infty} \alpha^{q_\infty+n-1} \int\limits_1^\infty \phi^{\tfrac{n+1+2r_\infty}{2}-1}e^{-n\f{p}(\alpha)\phi}\, d\phi\, d\alpha\\
& = \int\limits_1^{\infty} \alpha^{q_\infty+n-1}\frac{\Gamma(\frac{n+1+2r_\infty}{2},n\f{p}(\alpha))}{(n\f{p}(\alpha))^{\frac{n+1+2r_\infty}{2}}}\, d\alpha
\propto \int\limits_1^\infty \alpha^{q_\infty+n-2}e^{-n k \alpha}\, d\alpha < \infty,
\end{aligned}
\end{equation*}
where in the above $k\in \R^+$ is given in Proposition \ref{lim2b}. Therefore, from $s_i(\boldsymbol{x})<\infty, i=1,\ldots,4$, we have that $d=s_1(\boldsymbol{x})+s_2(\boldsymbol{x})+s_3(\boldsymbol{x})+s_4(\boldsymbol{x})<\infty$.

\subsection{Proof of Theorem \ref{fundteo2}}\label{ctheoremab3}

Suppose that hypothesis of item $i)$ hold.

First suppose that $\pi(\mu)  \gtrsim \mu^{k}$ with $k < -1$. Denoting $h=\sqrt{\frac{-k-1}{2n}}> 0$, it follows that for $0< \alpha\leq h$ and $0< \phi\leq h$ we have that $n\alpha\phi + k\leq nh^2 + k = \frac{(k- 1)}{2}  < -1$. Moreover, for every $\alpha>0$ fixed we have that $\exp\left\{-\mu^{\alpha}\sum_{i=1}^n x_i^\alpha\right\}\underset{\mu\to 0^{+}}{\propto} 1$, hence, from Proposition \ref{proposition1} we have that
\begin{equation*}
\begin{aligned}
\int\limits_0^\infty \pi(\mu)\mu^{n\alpha\phi}\exp{\left\{-\mu^{\alpha}\sum_{i=1}^n x_i^\alpha\right\}} \, d\mu
\gtrsim \int\limits_0^\infty \mu^{n\alpha\phi+k} = \infty,
\end{aligned}
\end{equation*}
for all fixed $\alpha\in (0,h]$ and $\phi\in (0,h]$. Therefore
\begin{equation*}
\begin{aligned}
d(\boldsymbol{x}) & \gtrsim  \int\limits_{h/2}^h \int\limits_{h/2}^h \pi(\alpha)\alpha^{n}\frac{\pi(\phi)}{\Gamma(\phi)^n}\left(\prod_{i=1}^n{x_i^{\alpha}}\right)^{\phi}\int\limits_0^{\infty}\mu^{n\alpha\phi+k}\exp{\left\{-\mu^{\alpha}\sum_{i=1}^n x_i^\alpha\right\}} \, d\mu \, d\phi \, d\alpha \\
& \propto \int\limits_{h/2}^h\int\limits_{h/2}^h \infty\; d\phi\; d\alpha = \infty,
\end{aligned}
\end{equation*}
that is, $d(\boldsymbol{x})=\infty$.

Now suppose that
 $\pi(\mu) \underset{\mu\to \infty}{\gtrsim} \mu^{k}$ and $\pi(\alpha) \underset{\alpha\to 0^+}{\gtrsim} \alpha^{q_0}$, where $k > -1$ and $q_0\in\R$. Under these hypothesis, in equation (\ref{weibullpos}) it was proved that
\begin{equation*}
\begin{aligned}
d(\boldsymbol{x};\phi)\propto\int\limits_0^\infty \int\limits_0^\infty \pi(\alpha)\alpha^{n}\left\{\prod_{i=1}^n{x_i^{\alpha\phi-1}}\right\}\pi(\mu)\mu^{n\alpha\phi}\exp\left\{-\mu^{\alpha}\sum_{i=1}^n x_i^\alpha\right\}\, d\mu\, d\alpha= \infty
\end{aligned}
\end{equation*}
for every $\phi>0$, and therefore
\begin{equation*}
\begin{aligned}
d(\boldsymbol{x})
& \propto \int\limits_0^\infty \frac{\pi(\phi)}{\Gamma(\phi)^n}\int\limits_0^\infty \int\limits_0^\infty \pi(\alpha)\alpha^{n}\left\{\prod_{i=1}^n{x_i^{\alpha\phi-1}}\right\}\pi(\mu)\mu^{n\alpha\phi}\exp\left\{-\mu^{\alpha}\sum_{i=1}^n x_i^\alpha\right\}\, d\mu\, d\alpha\, d\phi\\
&=\int_0^\infty\frac{\pi(\phi)}{\Gamma(\phi)^n}\cdot \infty\, d\phi = \infty
\end{aligned}
\end{equation*}
and thus $d(\boldsymbol{x}) = \infty$.

Suppose on the other hand that the hypotheses of ii) hold. Since $\pi(\mu)  \gtrsim \mu^{-1}$, following the same steps that resulted in (\ref{eqmain23}) and the same expressions for $s_i(\boldsymbol{x})$, where $i=1,\cdots, 4$, we have that $d(\boldsymbol{x}) \gtrsim s_1(\boldsymbol{x})+s_2(\boldsymbol{x})+s_3(\boldsymbol{x})+s_4(\boldsymbol{x})$. We now divide the proof that $d(\boldsymbol{x})=\infty$ in four cases:

\begin{itemize}
\item Suppose that $
\pi(\phi) \underset{\phi\to 0^+}{\gtrsim} \phi^{r_0}$ and $\pi(\alpha)
\underset{\alpha\to \infty}{\gtrsim} \alpha^{q_\infty}$ with $n\leq -r_0$. Then
\begin{equation*}
\begin{aligned}
s_2(\boldsymbol{x})
&\gtrsim \int\limits_1^{\infty}\alpha^{q_\infty+n-1}\int\limits_0^{1}\phi^{n+r_0-1}e^{-n\f{q}(\alpha)\phi} \, d\phi \, d\alpha \\
& = \int\limits_1^{\infty}\alpha^{q_\infty+n-1}\cdot \infty\, d\alpha = \infty
\end{aligned}
\end{equation*}
which implies $d(\boldsymbol{x})=\infty$.

\item Suppose that $
\pi(\phi) \underset{\phi\to 0^+}{\gtrsim} \phi^{r_0}$ and $\pi(\alpha)
\underset{\alpha\to \infty}{\gtrsim} \alpha^{q_\infty}$ with $q_\infty \geq r_0$ and $n> -r_0$. Then
\begin{equation*}
\begin{aligned}
s_2(\boldsymbol{x})
&\gtrsim \int\limits_1^{\infty}\alpha^{q_\infty+n-1}\int\limits_0^{1}\phi^{n+r_0-1}e^{-n\f{q}(\alpha)\phi} \, d\phi \, d\alpha \\
& = \int\limits_1^{\infty} \alpha^{q_\infty+n-1}\frac{\gamma(n+r_0,n\f{q}(\alpha))}{(n\f{q}(\alpha))^{n+r_0}}\, d\alpha
\propto \int\limits_1^{\infty}\alpha^{q_\infty-r_0-1}\, d\alpha = \infty
\end{aligned}
\end{equation*}
which implies $d(\boldsymbol{x})=\infty$.

\item Suppose that  $
\pi(\alpha) \underset{\alpha\to 0^+}{\gtrsim} \alpha^{q_0}$ and $ \pi(\phi) \underset{\phi\to \infty}{\gtrsim} \phi^{r_\infty}$
with $n\leq -q_0$. Then, by Proposition \ref{lim2b} we have that $q(\alpha)\underset{\alpha\to 0^+}{\propto} 0$ from where it follows that $e^{-nq(\alpha)\phi}\underset{\alpha\to 0^+}{\propto} 1$ and therefore
\begin{equation*}
\begin{aligned}
s_1(\boldsymbol{x})
& \gtrsim   \int\limits_0^{1} \pi(\phi)\phi^{n-1}\int\limits_0^{1}\alpha^{q_0+n-1}e^{-n\f{q}(\alpha)\phi}\, d\alpha \, d\phi  \\
& \propto  \int\limits_0^{1} \pi(\phi)\phi^{n-1}\int\limits_0^{1}\alpha^{q_0+n-1}\, d\alpha \, d\phi = \int\limits_0^{1} \pi(\phi)\phi^{n-1}\cdot \infty\, d\phi = \infty,
\end{aligned}
\end{equation*}
which implies $d(\boldsymbol{x})=\infty$.

\item Suppose that  $
\pi(\alpha) \underset{\alpha\to 0^+}{\gtrsim} \alpha^{q_0}$ and $ \pi(\phi) \underset{\phi\to \infty}{\gtrsim} \phi^{r_\infty}$
with $ 2r_\infty+1 \geq q_0$. Then
\begin{equation*}
\begin{aligned}
s_3(\boldsymbol{x})
&\gtrsim \int\limits_0^1 \alpha^{q_0+n-1} \int\limits_1^\infty \phi^{\tfrac{n+1+2r_\infty}{2}-1}e^{-n\f{p}(\alpha)\phi}\, d\phi\, d\alpha\\
& = \int\limits_0^1 \alpha^{q_0+n-1}\frac{\Gamma(\frac{n+1+2r_\infty}{2},n\f{p}(\alpha))}{(n\f{p}(\alpha))^{\frac{n+1+2r_\infty}{2}}}\, d\alpha
\propto \int\limits_0^1\alpha^{q_0-2r_\infty-2}\, d\alpha = \infty
\end{aligned}
\end{equation*}
which implies $d(\boldsymbol{x})=\infty$.
\end{itemize}
Therefore the proof is completed.
\end{document}